\documentclass[reqno]{amsart}
\usepackage[foot]{amsaddr}

\usepackage{caption}
\usepackage[hidelinks]{hyperref}
\usepackage{url}
\usepackage[capitalize,nameinlink]{cleveref}
\usepackage{xcolor}

\usepackage{amssymb}
\usepackage{amscd}
\usepackage{amsthm}
\usepackage{setspace}
\usepackage{enumerate}
\usepackage{graphicx}
\usepackage{mathtools}
\usepackage{tcolorbox} 
\usepackage{siunitx}
\usepackage{tikz-cd}
\usepackage{bm,blkarray}
\numberwithin{equation}{section}
\usepackage{algorithm}
\usepackage{algpseudocode}
\usepackage{subcaption}

\usepackage{mathtools}
\usepackage[tableposition=top]{caption}
\usepackage{booktabs,dcolumn}

\usepackage{pgfplots}
\pgfplotsset{compat=1.17}

\renewcommand\P{\mathbb{P}}

\newcommand\R{\mathbb{R}}

\newcommand{\rank}{\mathop{\rm rank}}

\usepackage{xcolor}      
\usepackage{listings}    

\definecolor{codegreen}{rgb}{0,0.6,0}
\definecolor{codegray}{rgb}{0.5,0.5,0.5}
\definecolor{codepurple}{rgb}{0.58,0,0.82}
\definecolor{white}{rgb}{1,1,1}
\definecolor{juliablue}{rgb}{0.2,0.4,0.8}
\definecolor{juliared}{rgb}{0.8,0.2,0.2}

\lstdefinestyle{juliasstyle}{
    backgroundcolor=\color{white},   
    commentstyle=\color{codegreen},
    keywordstyle=\color{juliablue}\bfseries,
    numberstyle=\tiny\color{codegray},
    stringstyle=\color{codepurple},
    basicstyle=\ttfamily\footnotesize,
    breakatwhitespace=false,         
    breaklines=true,                 
    captionpos=b,                    
    keepspaces=true,                 
    numbers=left,                    
    numbersep=5pt,                  
    showspaces=false,                
    showstringspaces=false,
    showtabs=false,                  
    tabsize=2,
    frame=single,
    rulecolor=\color{black},
    title=\lstname,
    morekeywords={abstract, break, case, catch, const, continue, do, else, elseif, end, export, false, for, function, global, if, import, in, let, macro, module, quote, return, struct, true, try, type, using, while},
    sensitive=true
}

\title[LU Factorization of Discrete Random Matrices]{LU Factorization of Discrete Random Matrices}
\author{Samuel Orellana Mateo}
\address{Department of Mathematics, Duke University, Durham, North Carolina, 27708 USA.}
\email{samuel.orellanamateo@duke.edu}
\author{John Urschel}
\address{Department of Mathematics, Massachusetts Institute of Technology, Cambridge, MA, 02139 USA.}
\email{urschel@mit.edu}
\author{Nicholas West}
\email{npwest00@mit.edu}

\newtheorem{theorem}{Theorem}[section]
\newtheorem{definition}[theorem]{Definition}
\newtheorem{lemma}[theorem]{Lemma}

\newtheorem{open}[theorem]{Open Problem}

\newtheorem{corollary}[theorem]{Corollary}

\newtheorem{remark}[theorem]{Remark}

\begin{document}

\begin{abstract}
 We consider the probability that a discrete random matrix $M_n(\xi)$ is \emph{strongly non-singular}, meaning all its leading principal submatrices are non-singular. This property is equivalent to the existence of an LU factorization. We show that for any discrete random variable $\xi$ with finite support and $|\xi|_\infty < 1$, there is a constant probability that $M_n(\xi)$ is strongly non-singular with a growth factor bounded by $n^{5/2+\delta}$. Furthermore, we provide a tight asymptotic lower bound for this probability as $|\xi|_\infty \to 0$. Finally, we provide exact counts for strongly non-singular binary matrices up to $n=9$ and use these to derive improved upper bounds for the Bernoulli case. 
\end{abstract}

\maketitle

\section{Introduction}

Gaussian elimination, the process in which a
matrix is factored into the product of a lower and upper triangular matrix, is the oldest and most-used algorithm for solving a general linear system. The resulting factorization $A = LU$, where $L$ is lower unitriangular and $U$ is upper triangular, is unique. However, an LU factorization may not exist if a leading principal minor (excluding the determinant of the matrix itself) of $A$ is zero. In this case, a permutation of the rows and/or columns of $A$ or a suitable pre- and/or post-multiplication is needed, increasing computational costs. In addition, when performing Gaussian elimination in finite precision, round-off error can accumulate. The stability of the Gaussian elimination algorithm in finite precision is governed by the growth factor, defined by
\[ \mathrm{growth}(A) = \max\left\{ \|L\|_{\max}, \frac{\|U\|_{\max}}{\|A\|_{\max}} \right\},\footnote{The exact definition of the growth factor varies by source. The quantity $\max_{k} \|A^{(k)}\|_{\max} /\|A\|_{\max}$ and the pair $\|L\|_\infty$ and $\|U\|_\infty/\|A\|_\infty$ are also popular, with the former implicitly assuming a pivoting strategy that produces bounded entries in $L$. They are all equivalent in that they can all be used to bound the backward error of the LU factorization.} \]
where $\|\cdot\|_{\max}$ is the entrywise matrix infinity norm. A matrix is said to be \emph{strongly non-singular} if all of its leading principal submatrices are non-singular. The existence of an LU factorization without pivoting for an $n \times n$ matrix is equivalent to the leading $(n-1)\times (n-1)$ principal submatrix being strongly non-singular.

In this paper, we consider an average case analysis of Gaussian elimination for discrete matrices. In particular, for a discrete random variable $\xi$, we consider the random matrix $A$ with entries sampled independently from $\xi$ and estimate the probability that $A$ has an LU factorization (essentially, the probability that it is strongly non-singular) and the probability that $A$ has a stable LU factorization in finite precision (i.e., a polynomially bounded $\mathrm{growth}(A)$). This analysis is closely connected to two seemingly disjoint areas of research, the numerical stability of Gaussian elimination and the singularity problem for discrete matrices, which we detail below.

\subsection{Numerical stability of Gaussian elimination} The numerical stability of Gaussian elimination has a long history, beginning, most notably, with the work of von Neumann and Goldstine \cite{von1947numerical} and later the work of Wilkinson \cite{wilkinson1961error,Wilkinson1965AEP}. Wilkinson was the first to recognize how the numerical stability of Gaussian elimination is related to the growth factor, and proved upper bounds for $\mathrm{growth}(A)$ under partial pivoting \cite[pg. 212]{Wilkinson1965AEP} and complete pivoting \cite{wilkinson1961error}. His upper bound of $2^{n-1}$ for partial pivoting is tight, and the large majority of research in this area is devoted to understanding the extent to which exponential growth is rare (see \cite{huang2024average,trefethen2012smart,trefethen1990average} and \cite[Lecture 22]{trefethen2022numerical} for more details regarding the average case and smoothed analysis of Gaussian elimination with partial pivoting). Wilkinson's quasi-polynomial upper bound for the growth factor under complete pivoting was quite pessimistic, but was only recently improved by an exponential constant \cite{bisain2023new}. A tight estimate for the growth factor under complete pivoting remains a major open problem. For more details on partial and complete pivoting, we refer the reader to \cite{urschel72numerical} and \cite[Section 1.1]{edelman2024some}. Significant attention has also been devoted to matrix transformations as an alternative to pivoting, see \cite{demmel2023improved,parker1995random,peca2023growth}.
In contrast, Gaussian elimination without pivoting is significantly easier to analyze. It is mostly understood in average case and smoothed analysis settings \cite{sankar2006smoothed,yeung1997probabilistic}. For instance, Sankar, Spielman, and Teng showed that, for a $n \times n$ Gaussian matrix $A$, $\mathbb{P}\left[\|U\|_\infty \ge M n^{3/2} \|A \|_{\infty} \right] \lesssim 1/M$ \cite[Theorem 4.3]{sankar2006smoothed} and $\mathbb{P}\left[\|L\|_\infty \ge M n^{2} \right] \lesssim 1/M$ \cite[Theorem 4.4]{sankar2006smoothed} (which implies that $\mathbb{P}[\mathrm{growth}(A) \ge  M n^{5/2}] \lesssim 1/M$). In Theorem \ref{thm:main}, we prove that, for discrete random matrices, the bound $\mathrm{growth}(A) \lesssim n^{5/2 + \delta}$ holds with constant probability for any $\delta >0$.

\subsection{Singularity of discrete random matrices} What is the probability that a binary matrix is singular? This folklore question has a long and rich history. Koml{\'o}s was the first to show that this probability was $o_n(1)$ \cite{komlos1967determinant}. A sequence of works by Kahn, Koml{\'o}s, and Szemer{\'e}di \cite{kahn1995probability}, Tao and Vu \cite{tao2005random,tao2007singularity}, and Bourgain, Vu, and Wood \cite{bourgain2010singularity} produced exponentially small bounds, with decreasingly small constants. Tikhomirov proved a tight bound of $(1/2 + o_n(1))^n$ \cite{tikhomirov2020singularity}. However, the stronger conjecture that this probability is $(1+o_n(1))n^2 2^{1-n}$ for a random $\{\pm 1\}$-matrix (see \cite[Equation 1]{kahn1995probability}) remains open. The work of Bourgain, Vu, and Wood also addresses more general classes of discrete random matrices, proving a bound of $(\sqrt{p}+o_n(1))^n$ for matrices with iid random entries $\xi$ with $\mathbb{P}[\xi = a] \le p$ for all $a$. The work of Jain, Sah, and Sawhney \cite{jain2021singularity} produced tight estimates for random matrices with non-uniform iid entries, and non-exact but improved estimates even in the uniform setting. Furthermore, the authors provided not only singularity estimates but bounds on the smallest singular value as well.

As noted above, the existence of an LU factorization is equivalent to all of the (strictly) leading principal submatrices being non-singular. As such, our analysis of the LU factorization of discrete matrices borrows many key results and techniques from this literature. In particular, the techniques of Subsection \ref{sub:n_moderate} are strongly inspired by \cite{komlos1967determinant}, and we make use of \cite[Theorem 1.4]{jain2021singularity} and \cite[Theorem 1.4]{bourgain2010singularity} (see Corollaries \ref{cor:jss} and \ref{cor:p_small}) to handle leading principal minors of arbitrarily large dimension.

\subsection{Our Contributions} In this work, we make three key contributions. For a discrete random variable $\xi$, let $|\xi|_{\infty}:= \max_{a \in \mathbb{R}} \mathbb{P} \left[\xi = a \right]$ and  $M_n(\xi)$ denote a $n \times n$ random matrix with entries independently distributed according to $\xi$. First, we prove that any discrete random matrix $M_n(\xi)$ has a constant probability of having a bounded growth factor:

\begin{theorem}\label{thm:main}
Let $\delta >0$ and $\xi$ be a discrete random variable with finite support and $|\xi|_{\infty}<1$. There exists a fixed constant $C_{\xi}<1$ such that
\[ \mathbb{P}\left[\mathrm{growth}(M_{n}(\xi)) > n^{\frac{5}{2} + \delta} \right] \le C_{\xi} + o_{n}(1).\]
\end{theorem}

The exponent of $5/2 + \delta$ is most likely not tight, and experiments suggest that the correct exponent for $\mathrm{growth}(A)$ (for both continuous and discrete random matrices) is $3/2$. Let $\Delta_k(A)$ denote the leading $k \times k$ principal minor of $A$ and $[n]:=\{1,\ldots,n\}$. Recall, an $n \times n$ matrix $A$ is said to be strongly non-singular if $\Delta_k(A) \ne 0$ for all $k \in [n]$. We note that, as a corollary, $M_n(\xi)$ has a constant probability of being strongly non-singular, or, equivalently, $M_{n+1}(\xi)$ has a constant probability of having an LU factorization:

\begin{corollary}\label{corr:main}
Let $\xi$ be a discrete random variable with finite support and $|\xi|_{\infty}<1$. There exists a fixed constant $C'_{\xi}$ such that
\[\mathbb{P}\left[M_n(\xi) \text{ is strongly non-singular}\right] \ge C'_{\xi} >0.\]
\end{corollary}

Of course, as $|\xi|_{\infty}$ tends to zero, the probability of a singular leading principal minor also tends to zero. Second, we give the exact rate at which the singularity probability does so. Let $M_n(p)$ denote an arbitrary $n \times n$ random matrix with independent entries $\xi_{ij}$, where each $\xi_{ij}$ is a discrete random variable, possibly taking infinitely many values, with $|\xi_{ij}|_\infty\le p$ for all $i,j \in [n]$. We prove the following result:

\begin{theorem}\label{thm:p_small}
There exists a constant $\alpha>0$ such that
\[\mathbb{P}\left[M_n(p) \text{ is strongly non-singular} \right] \ge 1 - \frac{5}{3} p - O(p^{1+\alpha})\]
for all $p \in (0,1)$, and this bound is tight. In particular, for $\xi \sim \mathrm{Uniform}(\{2^j\}_{j=0}^{\lfloor1/p\rfloor} \cup \{0\})$,
\[\mathbb{P}\left[M_n(\xi) \text{ is strongly non-singular} \right] = 1 - \frac{5}{3} p - O(p^{1+\alpha}).\]
\end{theorem}

The tight example of a random variable $\xi$ uniformly sampled from zero and a geometric progression experimentally achieves the theoretical estimate of $1-\tfrac{5}{3} p$ quite quickly (see Figure \ref{fig:unif}). We note that the generality of Theorem \ref{thm:p_small} comes at a cost: because the entries of $M_n(p)$ are only independent and not identical, statements regarding minimum singular values (or, equivalently, growth factors) are not possible without significant restrictions on the behavior of the entries.
\begin{figure}
\centering
\includegraphics[width = 4.5in]{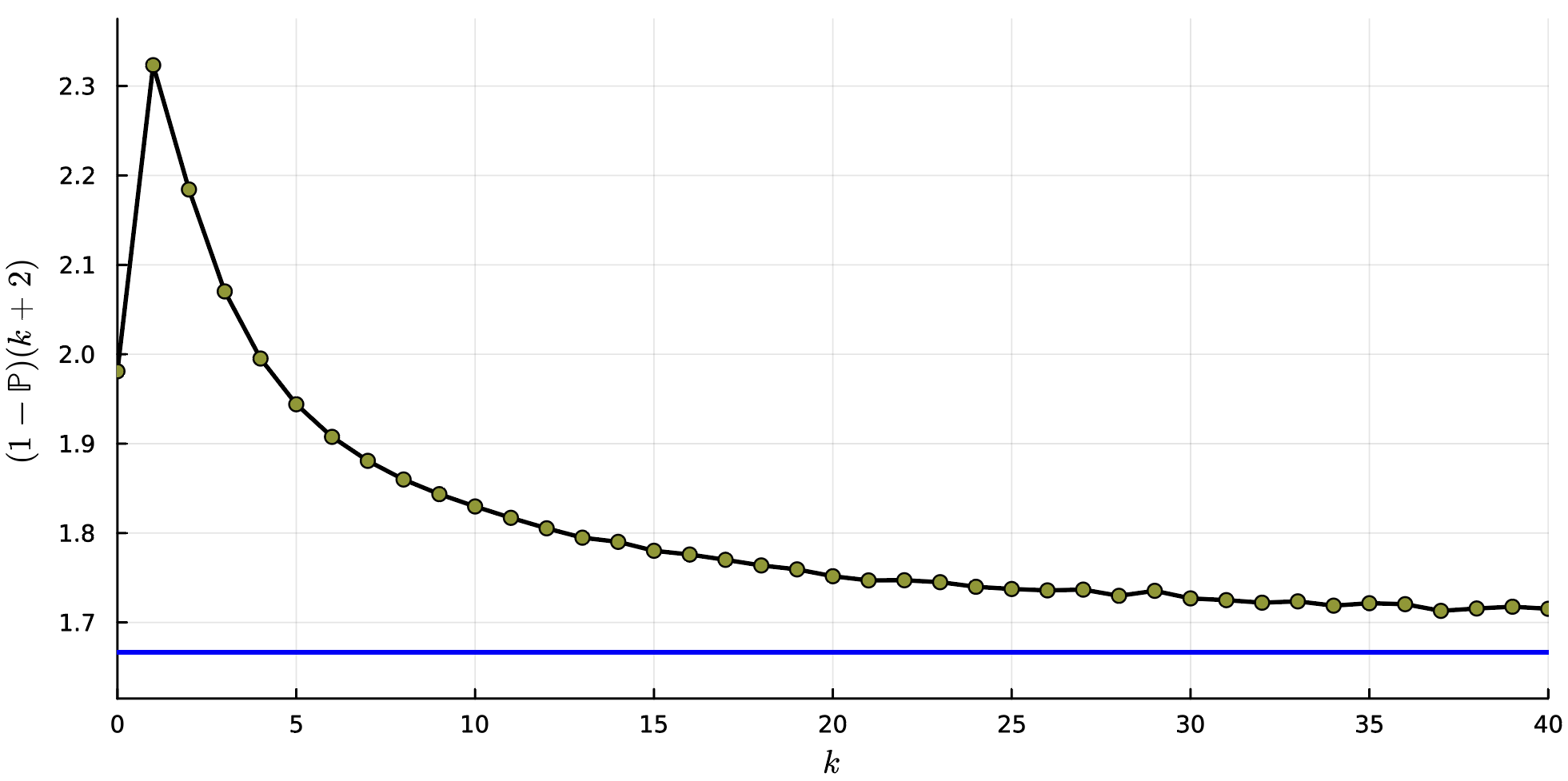}
\caption{The black line is empirical data for $\displaystyle{\frac{1-\mathbb{P}\big[M_n(\xi) \text{ is strongly non-singular}\big]}{|\xi|_\infty}}$, where $\xi \sim \mathrm{Uniform}(\{2^j\}_{j=0}^{k} \cup \{0\})$. The blue line is the value $5/3$, which Theorem \ref{thm:p_small} proves the black line approaches in limit. We see that the theoretical estimate of Theorem \ref{thm:p_small} is quite close to experimental results even for small $k$.}
\label{fig:unif}
\end{figure}

Finally, we pay special attention to the classical case of random $\{0,1\}$-matrices, i.e., $M_n(\xi)$ where $\xi \sim \mathrm{Ber}(p)$ equals $1$ with probability $p$, and $0$ with probability $1-p$. Corollary \ref{corr:main} provides an implicit constant lower bound for this setting. Here we provide exact counts for dimension at most nine, detailed experimental results for larger dimensions (in exact arithmetic), and theoretical upper bounds for sufficiently large dimensions. 

The exact number of strongly non-singular binary matrices, known as sequence OEIS A125587 by the Online Encyclopedia of Integer Sequences, was previously known up to dimension $6$, given by the sequence $$1, 4, 68, 5008, 1603232, 2224232640,$$ with $n = 1,\dots,4$ computed by Sloane and Vaishapayan and $n = 5,6$ computed by McKay \cite{OEISA125587}. We are not aware if the methods used by Sloane, Vaishapayan, and McKay are publicly available or if they differ from an exhaustive search over the $2^{n^2}$ binary matrices. 

In Section \ref{sec:sns}, we derive an algorithm that uses group theoretic and graph theoretic techniques to drastically improve upon exhaustive search; it relies on the construction of canonical representatives of strongly non-singular binary matrices which form equivalence classes under row and column permutations and an efficient search over the ``children'' of these matrices formed by appending a row and column. In addition to extending the above sequence to $n = 9$, it provides fine-grained sparsity data by reporting the number of strongly non-singular binary matrices of dimension $n$ with $k$ non-zero entries. These computations were performed on the Duke University Department of Computer Science cluster compute nodes, which are each equipped with 128 CPU threads and 1.15 TB of RAM. Utilizing a dedicated allocation of 100 CPU threads and 500 GB of RAM, evaluating the canonical equivalence classes up to dimension $8 \times 8$ required 30.5 days of computation, and computing the $9 \times 9$ extensions required an additional 7.2 days. We summarize our results in the following theorem:

\begin{figure}[t!]
\centering
  \begin{subfigure}[t]{0.325\textwidth}
    \centering
    \includegraphics[height =1.65 in]{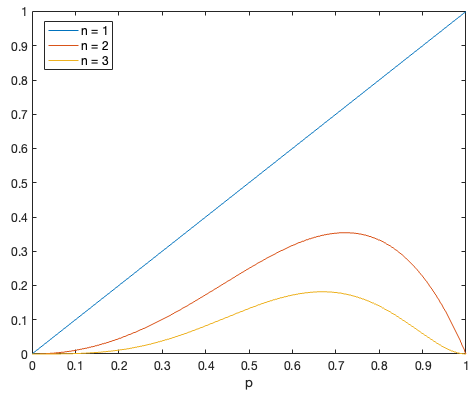}
    \caption{$n =1,\ldots,3$}
  \end{subfigure}
  \begin{subfigure}[t]{0.325\textwidth}
    \centering
\includegraphics[height =1.65 in]{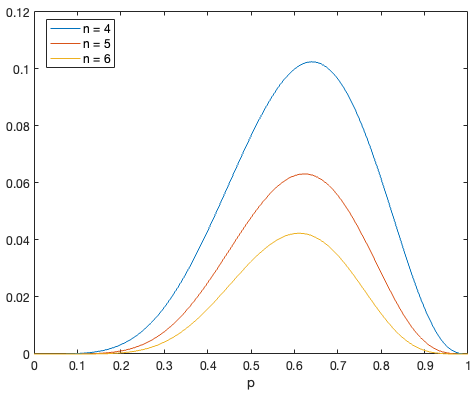}
    \caption{$n = 4,\ldots,6$}
  \end{subfigure}
  \begin{subfigure}[t]{0.325\textwidth}
    \centering
    \includegraphics[height =1.65 in]{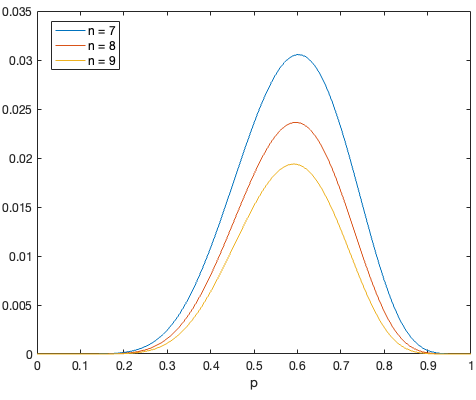}
    \caption{$n=7,\ldots,9$}
  \end{subfigure}
 \caption{Plots of $N_n(p):=\mathbb{P}\left[M_n(\mathrm{Ber}(p)) \text{ is strongly non-singular}\right]$ for $n = 1,\ldots,9$.}
\label{fig:exact_Nnp}
\end{figure}

\begin{theorem}
Let $N_n$ denote the number of strongly non-singular matrices in $\{0,1\}^{n\times n}$ and $N_{n,k}$ denote the number of strongly non-singular matrices in $\{0,1\}^{n\times n}$ with exactly $k$ non-zero entries. We have
\begin{align*}
N_7 &= 13\;288\;620\;860\;672, \\
N_8 &= 339\;148\;010\;221\;571\;072, \\
N_9 &= 36\;646\;054\;311\;185\;413\;881\;216,
\end{align*}
and the exact values of $N_{n,k}$ for all $n \le 9$ are reported in Tables \ref{tab:exact_counts} and \ref{tab:exact_counts_9x9}.
\end{theorem}

Using the values $N_{n,k}$ for $n \le 9$, we may exactly compute the functions \[N_n(p):=\mathbb{P}\left[M_n(\mathrm{Ber}(p)) \text{ is strongly non-singular}\right]\]
for $n \le 9$. We do so and plot the results in Figure \ref{fig:exact_Nnp}.

Exact values of strongly non-singular matrices for small dimensions can be extrapolated to upper bounds for larger dimensions. We perform this extrapolation using the data $N_{n,k}$ in dimension $n = 9$. 
\begin{figure}
\centering
\includegraphics[width = 4.5in]{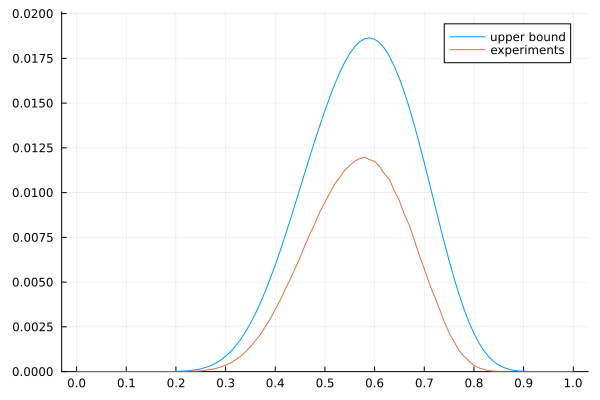}
\caption{Comparison of the analytical upper bound and experimental estimates for $N_{30}(p)$. The theoretical upper bound is evaluated from the exact $9 \times 9$ counts. Empirical values were generated via Monte Carlo simulation, testing $1001600$ random $30 \times 30$ extensions drawn from the set of strongly non-singular $4 \times 4$ base matrices.}
\label{fig:upper_exp}
\end{figure}
Our upper bounds are always within an additive factor of at most $0.68\%$ to what sampling suggests. For example, for a uniformly random $\{0,1\}$-matrix of dimension $n = 30$, we show that the probability of being strongly non-singular is bounded above by $1.45\ldots \%$, whereas experiments strongly suggest that this value is approximately $0.94\%$. Proving reasonably tight lower bounds (compared to the implicit constant in Corollary \ref{corr:main}) appears to be much more difficult. Such an improvement would likely require the creation of new techniques, as all known results that correctly capture the exponential rareness of singularity (i.e., the correct exponential rate) hide large, complicated constants. We leave the following question of providing effective lower bounds for this specific setting as an open problem:

\begin{open}
Prove that $\lim_{n \rightarrow \infty} \mathbb{P}\left[M_n(\mathrm{Ber}(\tfrac{1}{2})) \text{ is strongly non-singular}\right]  >0.5\%$.
\end{open}

\subsection{Remainder of paper} In Section \ref{sec:main_thm}, we prove Theorem \ref{thm:main}. In Section \ref{sec:anti_concentrated}, we prove Theorem \ref{thm:p_small}. In Section \ref{sec:sns}, we detail an algorithm to exactly count the number of $n \times n$ strongly non-singular binary matrices. In Section \ref{sec:experiments}, we detail our methodology for extrapolating our exact computations for dimension nine to upper bounds for $n \times n$ strongly non-singular binary matrices.

\section{LU Factorization of Discrete Random Matrices}\label{sec:main_thm}
Here we provide a proof of Theorem \ref{thm:main}. Our technique consists of handling small, medium, and large values of $n$ separately, and then relating singular value estimates to growth factor estimates.

\subsection{Bounding $n$ small} A very simple argument provides a loose, but sufficiently good first estimate for small values of $n$.

\begin{lemma}\label{lm:simple}
\[\mathbb{P}\left[\Delta_{k}(M_n(\xi)) \ne 0 \; \text{for all } k \in [n] \right] \ge (1 - |\xi|_\infty)^n\]
\end{lemma}

\begin{proof}
    We proceed by induction. By definition, $\P[\xi \ne 0] \ge 1- |\xi|_\infty$. Now, suppose $\Delta_k(M_n(\xi)) \ne 0$ for all $k \in [\ell]$ with probability at least $(1-|\xi|_\infty)^{\ell}$. Assuming that $\Delta_k(M_n(\xi)) \ne 0$ for all $k \in [\ell]$, the quantity $\Delta_{\ell+1}(M_n(\xi))$ is linear in the entry $[M_n(\xi)]_{\ell+1,\ell+1}$ with a non-zero coefficient, and therefore  $\Delta_{\ell+1}(M_n(\xi))$ can equal zero with probability at most $|\xi|_\infty$.
\end{proof}

\subsection{Bounding $n$ moderate}\label{sub:n_moderate} Here we apply a version of a technique introduced by K{\'o}mlos \cite{komlos1967determinant}, to produce a bound on invertibility that decays slowly. For the remainder of the subsection, let us write $M_{n+1}(\xi)$ in block notation as
\[M_{n+1}(\xi) = \begin{pmatrix} A & \bm{x} \\ \bm{y}^T & \alpha \end{pmatrix}, \qquad \text{where} \quad A \in \mathbb{R}^{n \times n},\; \bm{x}, \bm{y} \in \mathbb{R}^{n}, \; \alpha \in \mathbb{R}. \]
Conditional on the invertibility of $A$, $M_{n+1}(\xi)$ is singular if and only if $\alpha = \bm{y}^T A^{-1} \bm{x}$. We aim to bound $\mathbb{P}_{\bm{x},\bm{y},\alpha}\left[\alpha = \bm{y}^TA^{-1} \bm{x} \right]$ for a fixed invertible $A$, by breaking our analysis into two cases: either the number of non-zeros of $A^{-1} \bm{x}$, denoted $\mathrm{nnz}(A^{-1} \bm{x})$, is small, which is rare, or $\mathrm{nnz}(A^{-1} \bm{x})$ is large, which means it's unlikely that $\bm{y}^TA^{-1} \bm{x} = \alpha$. 

First, we produce a bound on the probability of the random vector $\bm{x} \in \mathbb{R}^{n}$ lying in some fixed subspace of a given dimension.

\begin{lemma}\label{lm:subspace_ac}
Let $W$ be a fixed $d$-dimensional subspace of $\mathbb{R}^{n}$ and $\bm{x} \in \mathbb{R}^{n}$ be a random vector with iid entries sampled from a discrete random variable $\xi$. Then $\mathbb{P}[\bm{x} \perp W] \le |\xi|_\infty^d$.
\end{lemma}
\begin{proof}
Let $\{\bm{w}_1, \bm{w}_2, \dots, \bm{w}_d\}$ be an orthonormal basis for $W$ and $Q = \left( \bm{w}_1 \; \bm{w}_2 \; \ldots \; \bm{w}_d\right) \in \mathbb{R}^{n \times d}$. The condition that $\bm{x}$ is orthogonal to $W$ is equivalent to $Q^T \bm{x} = 0$. The matrix $Q$ has rank $d$, and, without loss of generality, we may suppose that the first $d$ columns of $Q^T$ are linearly independent. Partitioning $Q^T$ and $\bm{x}$ as
\[
Q^T = \big(Q_L \;\; Q_R \big) \quad \text{and} \quad \bm{x} = \begin{pmatrix} \bm{x}_L \\ \bm{x}_R \end{pmatrix}, \quad \text{where} \; Q_L \in \mathbb{R}^{d \times d}, \; Q_R \in \mathbb{R}^{d \times (n-d)}, \; \bm{x}_L \in \mathbb{R}^d, \; \bm{x}_R \in \mathbb{R}^{n-d},
\]
the condition $Q^T \bm{x} = 0$ becomes $\bm{x}_L = - Q_L^{-1} Q_R \bm{x}_R$, i.e., the first $d$ entries of $\bm{x}$ are uniquely determined by the last $n-d$. The probability that $\bm{x}_L$ equals $- Q_L^{-1} Q_R \bm{x}_R$ is at most $|\xi|_\infty^d$, as desired.
\end{proof}

Using the above lemma, we can now effectively bound the probability that $\mathrm{nnz}(A^{-1} \bm{x})$ is small.

\begin{lemma}
Let $A \in \mathbb{R}^{n \times n}$ be a fixed invertible matrix, and $\bm{x} \in \mathbb{R}^{n}$ be a random vector with iid entries sampled from a discrete random variable $\xi$. Then
\[
\mathbb{P}\left[\operatorname{nnz}(A^{-1} \bm{x}) \le k \right] \le \exp\left\{k\log(n) +(n-k) \log |\xi|_\infty \right\} \qquad \text{for all} \quad k \ge 3.
\]
\end{lemma}\label{lm:sparse_rare}
\begin{proof}
For a subset $I \subset [n]$, let $E_I$ be the event
\[
E_I := \{(A^{-1} \bm{x})_i = 0 \text{ for all } i \in I\} = \{\bm{x} \perp \mathrm{span}\{\mathrm{row}_i(A^{-1})\, :\, i \in I \} \}.
\]
The event $\operatorname{nnz}(A^{-1} \bm{x}) \le k$ implies that there exists an index set $I$ of size $|I| = n-k$ for which $E_I$ occurs, or, equivalently, there is some $|I| = n-k$ for which $\bm{x}$ is orthogonal to the $(n-k)$-dimensional subspace $\mathrm{span}\{\mathrm{row}_i(A^{-1})\, :\, i \in I \}$. Using Lemma \ref{lm:subspace_ac} and a union bound,
\[
\mathbb{P}\left[\operatorname{nnz}(A^{-1} \bm{x}) \le k \right] \le \sum_{|I|=n-k} \mathbb{P}\left[E_I\right] \le \binom{n}{k} |\xi|_\infty^{n-k} \le  \left(\frac{en}{k}\right)^{k} |\xi|_\infty^{n-k}.
\]
Noting that $k \ge 3$, our desired result follows immediately
\begin{align*}  \mathbb{P}\left[\operatorname{nnz}(A^{-1} \bm{x}) \le k \right] &\le \exp\{k(\log n + 1 - \log k) + (n-k) \log |\xi|_\infty\} \\
     &\le  \exp\{k \log n + (n-k) \log |\xi|_\infty\}.
\end{align*} 
\end{proof}

When $\mathrm{nnz}(A^{-1} \bm{x})$ is not too small, we can make use of the following Littlewood-Offord type inequality of Rogozin to conclude that $\alpha = \bm{y}^TA^{-1} \bm{x}$ is rare.

\begin{corollary}[Corollary of {\cite[Theorem 2]{rogozin1961estimate}}]\label{cor:rogozin}
Let $\bm{z} \in \mathbb{R}^m$ be a fixed vector with $k$ non-zero entries, and $\bm{y} \in \mathbb{R}^{m}$ be a random vector with iid entries sampled from a discrete random variable $\xi$. Then
\[\sup_{a \in \mathbb{R}} \mathbb{P}\left[ \bm{y}^T \bm{z} = a \right] \le \frac{C}{\sqrt{k(1-|\xi|_\infty)}}\]
for an absolute constant $C$.
\end{corollary}

Combining Lemma \ref{lm:sparse_rare} and Corollary \ref{cor:rogozin}, we can obtain a gentler decay of the probability of all leading principal minors being non-zero.

\begin{lemma}\label{lm:n_moderate}
Let $A \in \mathbb{R}^{n \times n}$ be a fixed non-singular matrix, and $\bm{x},\bm{y} \in \mathbb{R}^{n}$ and $\alpha \in \mathbb{R}$ have iid entries sampled from a discrete random variable $\xi$. Then there exists some $n_0 \in \mathbb{N}$ such that
\[\mathbb{P}\left[ \det  \begin{pmatrix} A & \bm{x} \\ \bm{y}^T & \alpha \end{pmatrix} = 0\right] \le  \frac{1}{n^{1/3}}  \qquad \text{for all } n \ge n_0.\]
\end{lemma}

\begin{proof}
We break our analysis into two cases depending on the size of $\mathrm{nnz}(A^{-1} \bm{x})$. By Lemma \ref{lm:sparse_rare}, $\mathrm{nnz}(A^{-1} \bm{x}) \le k$ with probability at most $\exp\{k \log n + (n-k) \log |\xi|_{\infty} \}$ (for $k \ge 3$). By Corollary $\ref{cor:rogozin}$, if $\mathrm{nnz}(A^{-1} \bm{x}) > k$, then $\begin{pmatrix} \bm{y}^T & \alpha \end{pmatrix} \begin{pmatrix} A^{-1} \bm{x} \\ -1 \end{pmatrix} = 0$ with probability at most $C ((k+2)(1-|\xi|_{\infty}))^{-1/2}$ for some fixed constant $C$. Altogether, we have that, for $k = \lfloor (n \log |\xi|^{-1}_{\infty})/(2\log(n|\xi|^{-1}_{\infty})) \rfloor$,
\begin{align*}\mathbb{P}\left[ \det  \begin{pmatrix} A & \bm{x} \\ \bm{y}^T & \alpha \end{pmatrix} = 0\right] &\le \exp\{k \log n + (n-k) \log |\xi|_{\infty} \} + \frac{C}{\sqrt{(k+2)(1-|\xi|_{\infty})}}\\
&\le \exp\{ -\tfrac{1}{2} n \log |\xi|^{-1}_{\infty}\} + \frac{C \sqrt{2 \log(n |\xi|^{-1}_\infty)}}{\sqrt{(1-|\xi|_{\infty}) \, n \log |\xi|^{-1}_\infty}}.\end{align*}
For $n$ sufficiently large, this quantity is at most $1/n^{1/3}$.
\end{proof}

\subsection{Bounding $n$ large} To treat $n$ sufficiently large and produce bounds on singular values, we recall the following corollary of a result of Jain, Sah, and Sawhney.

\begin{corollary}[Corollary of {\cite[Theorem 1.4]{jain2021singularity}}]\label{cor:jss} Let $\xi$ be a discrete random variable taking finitely many values. There exists a constant $C_\xi > 0$ such that for any fixed $\epsilon > 0$ and for all sufficiently large $n$ and all $t \geq 0$,
\[
\mathbb{P}[\sigma_{\min}(M_n(\xi)) \leq t/\sqrt{n}] \leq C_\xi t + (1+\epsilon)^n |\xi|_\infty^n.
\]
\end{corollary}

\subsection{From singular values to growth factor}

The standard technique for estimating the growth factor of a random matrix is through estimates for the smallest singular values of leading sub-matrices. The below lemma does not exactly appear anywhere in the literature, and so we provide a short proof for completeness.

\begin{lemma}\label{lm:growth}
Let $A \in \mathbb{R}^{n \times n}$. Then
\[\mathrm{growth}(A) \le 1 + \|A\|_{\max} \max_{k \in [n-1]} k \, \sigma_{\min}^{-1}(A_{[k],[k]}).\]
\end{lemma}
\begin{proof}
Let $A_{S,T}$ denote the submatrix of $A$ indexed by the rows $S \subset [n]$ and $T \subset [n]$, and $A = LU$ be the LU factorization of $A$, given by
\[L_{ij} = \begin{cases} 
\displaystyle{\frac{\det(A_{[j-1]\cup\{i\},[j]})}{\det(A_{[j],[j]})}} & i \ge j  \\
0 & i<j 
\end{cases} \qquad \text{and} \qquad U_{ij} = \begin{cases} 
A_{1j} & i = 1 \\
 A_{ij} - A_{i,[i-1]} A_{[i-1],[i-1]}^{-1} A_{[i-1],j} & 1<i \le j \\
0 & i >j
\end{cases}.\]
We may easily bound the largest magnitude entry of $U$ using the smallest singular value
\[\|U\|_{\max} \le \max_{k \in [n-1]} \|A\|_{\max} + k \|A\|_{\max}^2 \, \sigma_{\min}^{-1}(A_{[k],[k]}). \]
For an entry $L_{ij}$, $i>j$, we have
\[L_{ij} = \frac{\det(A_{[j-1]\cup i,[j]})}{\det(A_{[j],[j]})} = \sum_{k =1 }^j A_{ik} (-1)^{j+k} \frac{\det(A_{[j-1],[j]\backslash k})}{\det(A_{[j],[j]})} = \sum_{k=1}^j A_{ik} [A_{[j],[j]}]^{-1}_{kj} = A_{i,[j]} A_{[j],[j]}^{-1} \bm{e}_j,\]
implying that
\[\|L\|_{\max} \le \max \left\{ \|A\|_{\max} \max_{k \in [n-1]} \sqrt{k} \sigma_{\min}^{-1}(A_{[k],[k]}),1 \right\}. \]
From here, our desired result follows immediately
\[\mathrm{growth}(A) = \max\left\{ \|L\|_{\max}, \frac{\|U\|_{\max}}{\|A\|_{\max}} \right\} \le 1 + \|A\|_{\max} \max_{k \in [n-1]} k \, \sigma_{\min}^{-1}(A_{[k],[k]}). \]
\end{proof}

\subsection{Proof of Theorem \ref{thm:main}}
Using Lemmas \ref{lm:simple} and \ref{lm:n_moderate} and Corollary \ref{cor:jss}, we are now prepared to prove Theorem \ref{thm:main}.

\begin{proof}[Proof of Theorem \ref{thm:main}]
By Lemmas \ref{lm:simple} and \ref{lm:n_moderate}, a lower bound for the probability of having all non-vanishing leading principal minors for $n> n_0$ is given by
\[(1-|\xi|_\infty)^{n_0} \prod_{k = n_0+1}^{n} \left(1 - \frac{1}{k^{1/3}}\right) \ge (1-|\xi|_\infty)^{n_0} \exp\{-2 \, n^{2/3}\}.\]
Choose $0 < \delta' < \delta$. Since $|\xi|_\infty < 1$, there exists an $\epsilon > 0$ such that $\gamma := (1+\epsilon)|\xi|_\infty < 1$. For some sufficiently large $n_1 > n_0$ and sufficiently small constant $\hat c$, Corollary \ref{cor:jss} is valid for this $\epsilon$ and $t = \hat c \, k^{-1-\delta'}$, and we have 
\[
\sum_{k = n_1}^\infty \gamma^{k}  + \sum_{k = n_1}^\infty \frac{\hat c \,  C_\xi }{k^{1+\delta'}} <  (1-|\xi|_\infty)^{n_0} \exp\{-2 \, {n_1}^{2/3}\}
\]
giving a lower bound of
\[
(1-|\xi|_\infty)^{n_0} \exp\{-2 \, {n_1}^{2/3}\} - \sum_{k = n_1}^\infty \gamma^{k} -  \sum_{k = n_1}^\infty \frac{\hat c \,  C_\xi }{k^{1+\delta'}} >0
\]
for the probability that all leading principal minors are non-zero and the smallest singular value of all the $k\times k$, $k \ge n_1$, leading principal minors are all at least $\hat c/k^{\frac{3}{2}+\delta'}$. Since $\xi$ is finitely supported, $\|A\|_{\max} \le M$ for some constant $M$. Applying Lemma \ref{lm:growth} gives $\mathrm{growth}(A) \le 1 + \frac{M}{\hat c} n^{\frac{5}{2}+\delta'}$. Taking $n$ large enough so that $1 + \frac{M}{\hat c} n^{\frac{5}{2}+\delta'} \le n^{\frac{5}{2}+\delta}$ and $n^{\frac{3}{2} + \delta'}$ is larger than the inverse of any of the smallest singular values of the leading principal minors of length at most $n_1$ completes the proof.
\end{proof}

\section{LU Factorization of Discrete Random Matrices with Anti-Concentrated Entries}\label{sec:anti_concentrated}

Here we provide a proof of Theorem \ref{thm:p_small}. Again, our technique consists of considering small $(n = 1,2)$, medium $n = O(\log 1/p)$, and large $n$ separately.

\subsection{Bounding $n =1,2$}

Analysis of dimension $n = 1$ is obvious, as $\mathbb{P}\left[\Delta_1(M_1(p)) = 0 \right] = \mathbb{P}\left[\xi_{11} = 0 \right] \le p$. However, already when $n = 2$, some work is required. First, we prove the following helpful lemmas regarding the collision probability $\sum_{\alpha} \mathbb{P}\left[ \xi = \alpha \right]^2$ for two i.i.d. copies of a discrete random variable $\xi$.

\begin{lemma}\label{lm:collision_prob}
Let $X, X',Y,Y'$ be independent discrete random variables taking finitely many values, with $X,\,X'$ and $Y,Y'$ each identically distributed. Then
\[\sum_{\alpha} \mathbb{P}\left[X+Y = \alpha\right]^2 \le \sqrt{\Big(\sum_{\beta} \mathbb{P}\left[X+X' = \beta\right]^2 \Big) \Big( \sum_{\gamma} \mathbb{P}\left[Y+Y' = \gamma\right]^2\Big)}\]
and $\sum_{\beta} \mathbb{P}\left[X+X' = \beta\right]^2$ is a convex function of the weights $\mathbb{P}[X= \alpha]$, $\alpha \in \mathrm{supp}(X)$.
\end{lemma}

\begin{proof}
Here, we are concerned only with probabilities involving pairwise sums of $X,X',Y,Y'$, and so, by the Frieman isomorphism lemma \cite[Lemma 5.25]{tao2006additive}, we may assume that $X,X',Y,Y'$ are all supported on $\mathbb{Z}$.
By Parseval's theorem and the discrete convolution theorem \cite[Section 2.9]{oppenheim1999discrete},
\begin{align*}
    \sum_{\alpha} \mathbb{P}\left[X+Y = \alpha\right]^2 &= \frac{1}{2 \pi} \int_{0}^{2 \pi} \Big| \sum_{\alpha} \mathbb{P}\left[X+Y = \alpha\right] e^{i \alpha \theta} \Big|^2 \, d \theta \\
    &= \frac{1}{2 \pi} \int_{0}^{2 \pi} \Big| \sum_{\beta} \mathbb{P}\left[X= \beta \right] e^{i \beta \theta} \Big|^2  \Big| \sum_{\gamma} \mathbb{P}\left[Y = \gamma \right] e^{i \gamma \theta} \Big|^2 \, d \theta \\
    &\le \sqrt{\Big( \frac{1}{2 \pi} \int_{0}^{2 \pi} \Big| \sum_{\beta} \mathbb{P}\left[X= \beta \right] e^{i \beta \theta} \Big|^4 \, d \theta \Big) \Big(\frac{1}{2 \pi} \int_{0}^{2 \pi}  \Big| \sum_{\gamma} \mathbb{P}\left[Y = \gamma \right] e^{i \gamma \theta} \Big|^4 \, d \theta  \Big)} \\
    &= \sqrt{\Big( \frac{1}{2 \pi} \int_{0}^{2 \pi} \Big| \sum_{\delta} \mathbb{P}\left[X + X' = \delta \right] e^{i \delta \theta} \Big|^2 \, d \theta \Big) \Big(\frac{1}{2 \pi} \int_{0}^{2 \pi}  \Big| \sum_{\epsilon} \mathbb{P}\left[Y + Y' = \epsilon \right] e^{i \epsilon \theta} \Big|^2 \, d \theta  \Big)} \\
    &= \sqrt{\Big(\sum_{\delta} \mathbb{P}\left[X+X' = \delta\right]^2 \Big) \Big( \sum_{\epsilon} \mathbb{P}\left[Y+Y' = \epsilon\right]^2\Big)}.
\end{align*}
Using the same argument,
\[\sum_{\delta} \mathbb{P}\left[X+X' = \delta\right]^2 = \frac{1}{2 \pi} \int_{0}^{2 \pi} \Big| \sum_{\beta} \mathbb{P}\left[X= \beta \right] e^{i \beta \theta} \Big|^4 \, d \theta.\]
From here, convexity follows quickly, as $\sum_{\beta} \mathbb{P}\left[X= \beta \right] e^{i \beta \theta}$ is a linear function of $\mathbb{P}\left[X= \beta \right]$, $|\cdot|^4$ is convex, the convex function of a linear function is convex, and integrals of convex functions are convex.
\end{proof}

\begin{lemma}\label{lm:grid}
Let $s_1<\ldots<s_n$ and $C_{ij} = \left| \{ (k,\ell) \, | \, s_i + s_j = s_k + s_{\ell}\}\right|$. Then
\[\phi(m):= \left| \{ (i,j) \, | \, C_{ij} \ge m \} \right| \le n^2 - m(m-1).\]
In addition, if $X$ and $X'$ are i.i.d. discrete random variables supported on $\le n$ points with $\max_{\alpha} \mathbb{P}\left[ X = \alpha \right] \le p$, then
\[\sum_{\alpha} \mathbb{P}\left[ X + X' = \alpha \right]^2 \le \frac{(2n^2+1)n p^4}{3}.\]
\end{lemma}

\begin{proof}
If $s_i \le s_j$, $s_{k} \le s_\ell$, $i \ne k$, and $s_i + s_j = s_k + s_{\ell}$, then either $i>k$ and $j<\ell$, or $i <k$ and $j > \ell$. Therefore, in order for $C_{ij} \ge m$, for $i\le j$, we must have
\[ \left| (n - j) - (i-1) \right| \le n -m. \]
Exactly $m(m-1)$ pairs fail to satisfy this condition for $m \le n$, and every pair fails to satisfy this condition for $m >n$. 

Now, consider $(X,X')$, supported on $\{s_1,\ldots,s_n\}^2$. We have $\mathbb{P}[(X,X') = (s_i,s_j)]\le p^2$ and
\[\sum_{\alpha} \mathbb{P}\left[ X + X' = \alpha \right]^2 \le \sum_{m=1}^n \frac{\phi(m)-\phi(m+1)}{m}\left(m p^2\right)^2 = p^4 \sum_{m=1}^n \phi(m). \]
By our above bound on $\phi(m)$,
\[\sum_{\alpha} \mathbb{P}\left[ X + X' = \alpha \right]^2 \le p^4 \sum_{m=1}^n n^2 - m(m-1)=\frac{(2n^2+1)n p^4}{3}.\]
\end{proof}

We are now prepared to provide a tight estimate for the case $n = 2$.

\begin{lemma}\label{lm:p_small_2}
\[\mathbb{P}\left[ \Delta_2(M_n(p)) = 0 \, | \, \Delta_1(M_n(p)) \ne 0 \right] \le \frac{2}{3} p +O(p^2)\]
\end{lemma}

\begin{proof}
Conditioning on $\Delta_1(M_n(p)) \ne 0$ is equivalent to replacing $\xi_{11}$ by a new random variable with $\xi'_{11}$ satisfying $|\xi'_{11}|_\infty \le p/(1-p) \le p + O(p^2)$. Therefore, we may safely ignore the conditioning in the lemma statement. In addition, we may also assume that the entries $\xi_{ij}$ are finitely supported, at negligible cost. The probability that any entry is zero and causes a vanishing minor is bounded by $O(p^2)$, allowing us to condition on all entries being non-zero. Let $\eta_{ij} = \log |\xi_{ij}|$. Because ignoring signs strictly upper bounds the probability of a multiplicative collision, we can apply Cauchy-Schwarz to obtain,
\begin{align*}
    \mathbb{P}\left[ \Delta_2(M_n(p)) = 0 \right] &= \sum_{\alpha} \mathbb{P}\left[\eta_{11}+\eta_{22} = \alpha \right] \mathbb{P}\left[\eta_{12}+\eta_{21} = \alpha \right] \\
    &\le \left( \sum_{\alpha} \mathbb{P}\left[\eta_{11}+\eta_{22} = \alpha\right]^2\right)^{1/2} \left( \sum_{\alpha} \mathbb{P}\left[\eta_{12}+\eta_{21} = \alpha\right]^2\right)^{1/2}, 
\end{align*} 
and so it suffices to bound $\sum_{\alpha} \mathbb{P}\left[\eta_{11}+\eta_{22} = \alpha\right]^2$. By Lemma \ref{lm:collision_prob}, we may instead bound the quantity $\sum_{\alpha} \mathbb{P}\left[\eta_{11}+\hat \eta_{11} = \alpha\right]^2$, where $\eta_{11}$ and $\hat \eta_{11}$ are independent and identically distributed. In addition, also by Lemma \ref{lm:collision_prob}, $\sum_{\alpha} \mathbb{P}\left[\eta_{11}+\hat \eta_{11} = \alpha\right]^2$ is a convex function of the weights $\mathbb{P}[\eta_{11} = \beta]$. Therefore, it is maximized at a boundary point, i.e., the weights of all support points equal either $p$ or $0$, save for a single support point, which has weight $1-\lfloor 1/p\rfloor p$. Applying Lemma \ref{lm:grid} with $\lceil 1/p \rceil$ support points completes the proof
\[\sum_{\alpha} \mathbb{P}\left[\eta_{11}+\hat \eta_{11} = \alpha\right]^2 \le \frac{(2\lceil 1/p \rceil^2+1)\lceil 1/p \rceil p^4}{3} = \frac{2}{3} p + O(p^2).\]
\end{proof}

\subsection{Bounding $n$ moderate} When $n = O(\log 1/p)$ and greater than $2$, we may make use of the following generalization of the Elkes-Szab{\'o} theorem for the intersection of semi-algebraic varieties and grids to higher dimensions.

\begin{theorem}[{\cite[Theorem A]{chernikov2024model}}]\label{thm:semi_alg}
Let $s \ge 4$, $Q \subset \R^{s}$ be semi-algebraic, of description complexity $D$, and such that the projection of $Q$ to any $s-1$ coordinates is finite-to-one. Then exactly one of the following holds:
\begin{enumerate}
    \item There exists a constant $c$ depending only on $s,D$ such that: for any $n \in \mathbb{N}$ and finite $S_i \subset \R$ with $|S_i| = n$ for $i \in [s]$ we have $|Q \cap (S_1 \times \ldots \times S_s)| \le c n^{s-4/3}$. 
    \item There exists open sets $U_i \subset \mathbb{R}$, $i \in [s]$, an open set $V \subset \mathbb{R}$ containing $0$, and analytic bijections with analytic inverses $\pi_i:U_i \rightarrow V$ such that
   $\pi_1(x_1) + \ldots + \pi_s(x_s) = 0 \, \Leftrightarrow \, Q(x_1,\ldots,x_s)$
    for all $x_i \in U_i$, $i \in [s]$.
\end{enumerate}
\end{theorem}

In particular, we aim to apply the above theorem to the intersection of the set of singular $3\times 3$ matrices and a grid $ \prod_{i,j=1}^3 S_{ij} \subset \mathbb{R}^{3 \times 3}$. First, we prove the following result regarding a $C^2$ function whose zero level set locally matches that of a linear function.

\begin{lemma}\label{lm:linear}
Let $f$ be $C^2$ in an open neighborhood of $0 \in \mathbb{R}^n$, $g$ be a non-zero linear function on $\mathbb{R}^n$, $f(0) = g(0) = 0$, $\nabla f(0) \ne 0$, and, in an open neighborhood of $0$, $f(x) = 0$ if and only if $g(x) = 0$. Then there exists a $C^1$ function in an open neighborhood of $0$, $u$, with $f(x) = g(x) u(x)$ and $u(0) \ne 0$.
\end{lemma}

\begin{proof}
Because $g$ is linear, there is a change of basis $\Phi$ of $x$ such that, for $\alpha \in \mathbb{R}$ and $y \in \mathbb{R}^{n-1}$, $\Phi(\alpha,y) = x$ and $\alpha = g(x)$. Let $F(\alpha,y) = f(\Phi(\alpha,y))$. Near the origin, we have $F(0,y) = 0$, $F(0,0) = 0$, $\nabla F(0,0) \ne 0$, $F^{-1}(0) = \{\alpha = 0\}$, and 
\[F(\alpha,y) = F(\alpha,y) - F(0,y) = \int_{0}^\alpha \frac{\partial F}{\partial \alpha} (\beta,y) \, d \beta = \alpha \int_{0}^1 \frac{\partial F}{\partial \alpha} (\tau \alpha, y) \, d \tau. \]
Let $w(\alpha,y) = \int_{0}^1 \frac{\partial F}{\partial \alpha} (\tau \alpha, y) \, d \tau$ and note that, because $F$ is $C^2$ near the origin, $w$ is also $C^1$ near the origin. Letting $u(x) = w(g(x),\Phi^{-1}(x)|_{y})$, we have
\[f(x) = F(g(x),\Phi^{-1}(x)|_{y}) = g(x) u(x). \]
Finally, we note that
\[u(0) = \int_{0}^1\frac{\partial F}{\partial \alpha} (0, 0) \, d \tau = \frac{\partial F}{\partial \alpha} (0, 0)\]
and the tangent space of the zero level set at the origin is $\{ (\alpha, y) \, | \, \alpha = 0\}$, therefore, $\frac{\partial F}{\partial \alpha} (0, 0) \ne 0$.
\end{proof}

\begin{lemma}\label{lm:det_grid}
Let $\mathcal{S}^* = \{A \in \mathbb{R}^{3 \times 3} : \det(A) = 0 \text{ and no } 2 \times 2 \text{ minor is zero}\}$. There exists a constant $c$ such that, for any $S_{ij} \subset \mathbb{R}$ with $|S_{ij}| = n$ for $i,j  \in [3]$, we have $\big|\mathcal{S}^* \cap \big(\prod_{i,j=1}^3 S_{ij} \big) \big| \le c n^{23/3}$.
\end{lemma}

\begin{proof}
We apply Theorem \ref{thm:semi_alg} to the semi-algebraic set $\mathcal{S}^*$. The determinant is linear in any entry, and because no $2 \times 2$ minor is zero for matrices in $\mathcal{S}^*$, fixing any 8 entries of $A \in \mathcal{S}^*$ leaves a non-trivial linear equation for the 9th entry, which has one solution. Thus, the projection of $\mathcal{S}^*$ to any 8 coordinates is at finite-to-one, satisfying the hypothesis of Theorem \ref{thm:semi_alg}.

It suffices to show that $\mathcal{S}^*$ does not satisfy property (2) of Theorem \ref{thm:semi_alg}. Let us suppose property (2) holds. Note that there must exist a matrix $\hat A \in U_1 \times \ldots \times U_9$ with $\rank( \hat A) = 2$ and no vanishing $2 \times 2$ minors, as the subset of matrices violating these conditions has dimension at most 7, while $\mathcal{S}^*$ has dimension 8. Let $B = \pi(A)$, with, without loss of generality, $\pi(\hat A) = 0$. For $B \in V$, $f(B) := \sum_{i,j = 1}^3 B_{ij} =0$ if and only if $\det(\pi^{-1}(B)) = 0$. Because $\rank( \hat A) = 2$, $\nabla \det(\pi^{-1}(0)) \ne 0$, and, by Lemma \ref{lm:linear}, $\det(\pi^{-1}(B)) = f(B) u(B)$ for some $C^1$ function $u$ with $u(0) \ne 0$. Computing the mixed second derivatives of both sides, we find that
\[\frac{\partial^2 \det(\pi^{-1}(B))}{\partial B_{ij} \partial B_{k\ell}} (0) = \frac{\partial u}{\partial B_{ij}}(0) + \frac{\partial u}{\partial B_{k \ell}}(0) \qquad \text{for} \quad (i,j) \ne (k,\ell).\]
However, by the linearity of the determinant in each row and column, we have $\frac{\partial^2 \det(\pi^{-1}(B))}{\partial B_{ij} \partial B_{\ell k}} (0) = 0$ when $(i,j)$ and $(k,\ell)$ are in the same row or column. The resulting system of equations in derivatives of $u$ implies that $\frac{\partial u}{\partial B_{ij}}(0) = 0$ for all $(i,j)$, and therefore all mixed derivatives of $\det(\pi^{-1})$ vanish at $0$. However, this is a contradiction, as $\hat A$ has rank two, and so $\frac{\partial^2 \det(\hat A)}{\partial A_{ij} \partial A_{k,\ell}}$ is non-zero for some $(i,j) \ne (k,\ell)$. Because $\pi_{ij}'(0) \ne 0$ for all $(i,j)$, the chain rule implies that the mixed derivatives of $\det(\pi^{-1}(B))$ cannot all vanish.
\end{proof}

\begin{remark}
The proof of Lemma \ref{lm:det_grid} crucially requires that $n \ge 3$. Note that, by Lemma \ref{lm:p_small_2}, an analogous statement is not true for $n = 2$. In particular, while the mixed derivatives in each row and column do vanish, because there are only two entries in each row and column, this is not sufficient to conclude that all first derivatives vanish. For example, for an entry-wise positive matrix in  $\mathbb{R}^{2 \times 2}$, logarithmic bijections recognize singularity:
\[\det(A) = 0 \; \text{ for } A \in (0,\infty)^{2 \times 2} \; \Longleftrightarrow \; \log(A_{11}) - \log (A_{12}) - \log (A_{21}) + \log(A_{22}) = 0.\]
Now, taking 
\[B = \begin{pmatrix} \phantom{-}\log(A_{11}) & - \log (A_{12}) \\ - \log(A_{21}) & \phantom{-}\log(A_{22}) \end{pmatrix}\]
and $\hat A = \bm{1} \bm{1}^T$, the mixed derivatives of $\det (\pi^{-1}(B)) = e^{B_{11} + B_{22}} - e^{-(B_{12} + B_{21})}$ along a row or column do vanish, but the first derivatives of
\[u(B) = \frac{e^{B_{11} + B_{22}} - e^{-(B_{12} +B_{21})}}{B_{11} + B_{12} + B_{21} + B_{22}},\]
appropriately defined at indefinite points,
do not vanish at $B = 0$: $\partial u/\partial B_{11} (0) = \partial u/\partial B_{22} (0) = 1/2$ and $\partial u/\partial B_{12} (0) = \partial u/\partial B_{21} (0) = -1/2$.
\end{remark}

Using the above lemma, we can now bound $\mathbb{P}[\Delta_n(M_n(p)) = 0\, |\, \Delta_{n-3}(M_n(p)) \ne 0]$.

\begin{lemma}\label{lm:p_small_med_n}
Let $n \ge 3$. Then $\mathbb{P}[\Delta_n(M_n(p)) = 0\, |\, \Delta_{n-3}(M_n(p)) \ne 0] = O(p^{27/26})$.
\end{lemma}

\begin{proof}
Let $A = M_n(p)$ be written in block notation as
\begin{align*}
A &=  \begin{pmatrix} B & \bm{x} & \bm{y} & \bm{z} \\ \bm{u}^T & \xi_{11} & \xi_{12} & \xi_{13} \\ \bm{v}^T & \xi_{21} & \xi_{22} & \xi_{23} \\ \bm{w}^T & \xi_{31} & \xi_{32} & \xi_{33} \end{pmatrix} \\
&= \begin{pmatrix} I &  &  &  \\ \bm{u}^T B^{-1} & 1 &  & \\ \bm{v}^T B^{-1} &  & 1 & \\ \bm{w}^T B^{-1} &  & & 1\end{pmatrix} \begin{pmatrix} B & \bm{x} & \bm{y} & \bm{z} \\  & \xi_{11} - \bm{u}^T B^{-1} \bm{x} & \xi_{12} - \bm{u}^T B^{-1} \bm{y} & \xi_{13}- \bm{u}^T B^{-1} \bm{z} \\ & \xi_{21} - \bm{v}^T B^{-1} \bm{x}& \xi_{22} - \bm{v}^T B^{-1} \bm{y}& \xi_{23} - \bm{v}^T B^{-1} \bm{z} \\ & \xi_{31} - \bm{w}^T B^{-1} \bm{x}& \xi_{32} - \bm{w}^T B^{-1} \bm{y}& \xi_{33} - \bm{w}^T B^{-1} \bm{z} \end{pmatrix},
\end{align*}
where $B \in \mathbb{R}^{(n-3)\times(n-3)}$ and $\bm{u},\bm{v},\bm{w},\bm{x},\bm{y},\bm{z} \in \mathbb{R}^{n-3}$. Let us condition on a (non-singular) realization of $B$ and of the vectors $\bm{u},\bm{v},\bm{w},\bm{x},\bm{y},\bm{z}$, and consider the matrix
\[C = \begin{pmatrix}  \xi_{11} - \bm{u}^T B^{-1} \bm{x} & \xi_{12} - \bm{u}^T B^{-1} \bm{y} & \xi_{13}- \bm{u}^T B^{-1} \bm{z} \\  \xi_{21} - \bm{v}^T B^{-1} \bm{x}& \xi_{22} - \bm{v}^T B^{-1} \bm{y}& \xi_{23} - \bm{v}^T B^{-1} \bm{z} \\  \xi_{31} - \bm{w}^T B^{-1} \bm{x}& \xi_{32} - \bm{w}^T B^{-1} \bm{y}& \xi_{33} - \bm{w}^T B^{-1} \bm{z} \end{pmatrix}\]
Then, after conditioning, $C = M_3(p)$ and $\det(A)$ is non-zero if and only if $\det(C)$ is non-zero. Therefore, we may restrict ourselves to the unconditional case of $n = 3$ and the matrix $C$. We decompose our analysis into a variety of cases that, together, bound $\mathbb{P}[\det(C) = 0]$. We have
\[\mathbb{P}[C_{ij} = C_{\ell k} = 0 \text{ for } (i,j) \ne (\ell,k)] = O(p^2)\]
and
\[\mathbb{P}[C_{11} C_{22} = C_{12} C_{21} \ne 0 \text{ and } C_{12} C_{23} = C_{22} C_{13} \ne 0] = O(p^2),\]
and so the probability that $C$ has a vanishing minor of order two and $\det(C) = 0$ is $O(p^2)$. Now, we may restrict ourselves to realizations of $C$ with no vanishing minors of order two, and further divide our analysis based on whether any entry of $C$ equals a ``rare" value. Using the same argument from Lemma \ref{lm:simple},
\[\mathbb{P} \left[\det(C) = 0, \, C_{11}C_{22} \ne C_{12} C_{21},\, C_{33} = \alpha \text{ for } \alpha \text{ s.t. } \mathbb{P}[C_{33} = \alpha] \le p^{27/26} \right] = O(p^{27/26}).\]
If no entry of $C$ equals a ``rare" value, then we may simply use Lemma \ref{lm:det_grid} to obtain
\begin{align*}
\mathbb{P} \left[ \det(C) = 0, \, C_{ij} = \alpha_{ij}\text{ for } \alpha_{ij} \text{ s.t. } \mathbb{P}[C_{ij} = \alpha_{ij}]> p^{27/26} \text{ for } i,j = 1,2,3\right] &\le \left\lceil \frac{1}{p^{27/26}} \right\rceil^{23/3} p^{9} \\
&= O(p^{27/26}).
\end{align*}
\end{proof}

\subsection{Bounding $n$ large} When $n$ is sufficiently large relative to $\log 1/p$, we may simply make use of the following corollary of a result of Bourgain, Vu, and Wood.

\begin{corollary}[Corollary of {\cite[Theorem 1.4]{bourgain2010singularity}}]\label{cor:p_small}
There exists fixed constants $\delta \in (0,1)$ and $n_0 \in \mathbb{N}$ such that $\mathbb{P} \left[ \Delta_n(M_n(p)) = 0 \right] \le \delta^n$
for all $p \le 1/2$ and $n \ge n_0$.
\end{corollary}

\subsection{Proof of Theorem \ref{thm:p_small}} Using Lemmas \ref{lm:p_small_2} and \ref{lm:p_small_med_n}, and Corollary \ref{cor:p_small}, we are now prepared to prove Theorem \ref{thm:p_small}.

\begin{proof}[Proof of Theorem \ref{thm:p_small}]
Let $\ell = \max\left\{n_0, \left\lceil \frac{2 \log p}{\log \delta} \right\rceil \right\}$, where $n_0$ and $\delta$ are the constants of Corollary \ref{cor:p_small}. We have
\begin{align*}
\mathbb{P}\left[ \Delta_k(M_n(p)) = 0 \text{ for some } k \in [n]\right] &\le \mathbb{P} \left[ \Delta_1(M_n(p)) = 0 \right] + \mathbb{P}\left[ \Delta_2(M_n(p)) = 0 \, | \, \Delta_1(M_n(p)) \ne 0 \right]\\
&\quad + \sum_{k=3}^\ell \mathbb{P}\left[ \Delta_k(M_n(p)) = 0 \, | \, \Delta_{k-3}(M_n(p)) \ne 0 \right] \\
&\quad + \sum_{k =\ell+1}^n \mathbb{P} \left[ \Delta_k(M_n(p)) = 0 \right].
\end{align*}
Clearly, the first term is at most $p$. By Lemma \ref{lm:p_small_2}, the second term is at most $\frac{2}{3}p + O(p^2)$. By Lemma \ref{lm:p_small_med_n}, the third term is at most $O(\ell p^{27/26})$. Finally, by Corollary \ref{cor:p_small}, and noting that our choice of $\ell$ gives $\delta^\ell \le p^2$,
\[\sum_{k = \ell+1}^n \mathbb{P} \left[ \Delta_k(M_n(p)) = 0 \right] \le \sum_{k = \ell+1}^n \delta^k \le \frac{\delta^{\ell+1}}{1-\delta} \le \frac{\delta^\ell}{1-\delta} \le \frac{p^2}{1-\delta} = O(p^2).\]

To verify tightness, note $\mathbb{P}[\Delta_1(M_n(\xi)) = 0] = p + O(p^2)$. Given $\Delta_1(M_n(\xi)) \ne 0$, $\Delta_2(M_n(\xi)) = 0$ is dominated by the collision $\xi_{11}\xi_{22} = \xi_{12}\xi_{21}$ among non-zero entries. Taking base-$2$ logarithms reduces this directly to the uniform additive collision problem of \cref{lm:p_small_2}, giving a probability of $\frac{2}{3}p + O(p^2)$. Since higher-order minors contribute at most $O(p^{27/26})$, the probability of a singular leading principal minor is exactly $\frac{5}{3}p + O(p^{27/26})$.
\end{proof}

\section{Exact Counting of Strongly Non-Singular $\{0,1\}$-Matrices}
\label{sec:sns}

Here we detail an algorithm for exactly counting the number of $n \times n$ strongly non-singular $\{0,1\}$-matrices for $n = 1,2,\ldots,9$ via enumeration. The associated code and computational results are available at the GitHub repository \cite{OrellanaMateo2026Repo}. The naive enumeration of $n \times n$ strongly non-singular $\{0,1\}$-matrices is computationally intractable for $n \ge 9$, as the search space size $2^{n^2}$ grows super-exponentially. To address this, we implement a counting algorithm based on identifying equivalence classes of matrices under row and column permutations. This approach reduces the space sufficiently to compute exact counts up to $n=9$.

\subsection{Equivalence Classes and Recursive Construction}

We define an equivalence relation on the set of $n \times n$ binary matrices $\{0,1\}^{n \times n}$. Two matrices $A, B \in \{0,1\}^{n \times n}$ are equivalent, denoted $A \sim B$, if there exist permutation matrices $P$ and $Q$ such that $A = P B Q$ or $A = P B^T Q$. Let 
\[\Omega_n := \{ A \in \{0,1\}^{n \times n} | \det(\Delta_k(A)) \neq 0 \; \forall k \in [n] \}\]
be the set of $n\times n$ strongly non-singular $\{0,1\}$-matrices. Given the set of equivalence classes $\Omega_n/\sim$ and their sizes, we detail an algorithm to construct $\Omega_{n+1}/\sim$ and their sizes.

We say a matrix $C$ is an extension of $A \in \Omega_n$ if 
\[
C = \begin{pmatrix} A & \mathbf{x} \\ \mathbf{y}^T & \alpha \end{pmatrix}, \quad \text{where } \mathbf{x}, \mathbf{y} \in \{0,1\}^n, \alpha \in \{0,1\}.
\]
The matrix $C$ is strongly non-singular if and only if $A \in \Omega_n$ and $\det(C) = \alpha \cdot \det(A) -  \mathbf{y}^T \operatorname{adj}(A) \mathbf{x} \neq 0$. When this is the case, we say $C$ is a child of $A$, and denote the set of strongly non-singular extensions of a matrix $A$ by $\mathrm{Children}(A)$.

\begin{lemma}[Isomorphism Preservation]
Let $A, B \in \Omega_n$ and $A \sim B$. There exists a bijection $\phi: \mathrm{Children}(A) \to \mathrm{Children}(B)$ such that, for every child $C$ of $A$, $C \sim \phi(C)$.
\end{lemma}

\begin{proof}
By definition, there exist permutation matrices $P$ and $Q$ such that either $A = P B Q$ or $A = P B^T Q$. Without loss of generality, suppose $A = P BQ$ (the case $A = P B^T Q$ is very similar) and let $C$ be an extension of $A$:
\[
C = \begin{pmatrix} A & \mathbf{x} \\ \mathbf{y}^T & \alpha \end{pmatrix}.
\]

We define the corresponding extension of $B$ as $D = \begin{pmatrix} B & \mathbf{x}' \\ \mathbf{y}'^T & \alpha \end{pmatrix}$, where $\mathbf{x}' = P^T \mathbf{x}$ and $\mathbf{y}' = Q \mathbf{y}$. It suffices to show that $C$ is a child if and only if $D$ is a child. Note that $\det(A) = \det(P)\det(B)\det(Q) = \pm \det(B)$, and so
\begin{align*}
\det(D) &= \det(B)(\alpha - \mathbf{y}'^T B^{-1} \mathbf{x}') \\
&= \det(B)(\alpha - (Q\mathbf{y})^T (P^T A Q^T)^{-1} (P^T \mathbf{x})) \\
&= \det(B)(\alpha - \mathbf{y}^T Q^T (Q A^{-1} P) P^T \mathbf{x}) \\
&= \det(B)(\alpha - \mathbf{y}^T A^{-1} \mathbf{x}) = \pm \det(C).
\end{align*}
Thus, $D \in \Omega_{n+1}$ if and only if $C \in \Omega_{n+1}$. In addition, for extended permutations $\mathcal{P} = \begin{pmatrix} P & 0 \\ 0 & 1 \end{pmatrix}$ and $\mathcal{Q} = \begin{pmatrix} Q & 0 \\ 0 & 1 \end{pmatrix}$,
\[
\mathcal{P} D \mathcal{Q} = \begin{pmatrix} P & 0 \\ 0 & 1 \end{pmatrix} \begin{pmatrix} B & P^T \mathbf{x} \\ \mathbf{y}^T Q^T & \alpha \end{pmatrix} \begin{pmatrix} Q & 0 \\ 0 & 1 \end{pmatrix} 
= \begin{pmatrix} P B Q & P P^T \mathbf{x} \\ \mathbf{y}^T Q^T Q & \alpha \end{pmatrix} 
= \begin{pmatrix} A & \mathbf{x} \\ \mathbf{y}^T & \alpha \end{pmatrix} = C.
\]
Thus $C \sim D$, and the map $(\mathbf{x}, \mathbf{y}, \alpha) \mapsto (P^T\mathbf{x}, Q\mathbf{y}, \alpha)$ is a bijection between the extension spaces. The case $A = P B^T Q$ is similar, with corresponding bijection $(\mathbf{x}, \mathbf{y}, \alpha) \mapsto (Q\mathbf{y}, P^T\mathbf{x}, \alpha)$.
\end{proof}

This lemma allows us to process only one representative per equivalence class. Let $[A]$ be the image of $A$ in the map $\Omega_n \to \Omega_n / \sim$. If a representative $A$ has count $N_A$ (the size of its orbit times the number of previous parents), and produces a child $C$ with multiplicity $k$ (number of valid $(\mathbf{x}, \mathbf{y}, \alpha)$ extensions reducing to $C$), we add $k \cdot N_A$ to the count of the class of $C$. Thus we can construct $\Omega_{n+1} / \sim$ from $\Omega_n / \sim$. We detail this procedure in Algorithm \ref{alg:generation}.

\begin{algorithm}[t]
\caption{Generation of Strongly Non-Singular Representatives of Dimension $n+1$}
\label{alg:generation}
\begin{algorithmic}[1]
\Require $\Omega_n/\sim$: Set of canonical representatives of size $n \times n$ and their orbit counts $N_A$.
\Ensure $\Omega_{n+1}/\sim$: Set of canonical representatives of size $(n+1) \times (n+1)$ and counts.
\State Initialize an empty hash map $\mathcal{H}$ to store $\Omega_{n+1}/\sim$
\For{each $(A, N_A) \in \Omega_n/\sim$}
    \For{each extension $\mathbf{x}, \mathbf{y} \in \{0,1\}^n$ and $\alpha \in \{0,1\}$}
        \State $C \gets \begin{pmatrix} A & \mathbf{x} \\ \mathbf{y}^T & \alpha \end{pmatrix}$
        \If{$\det(C) \neq 0$} \Comment{Computed via Bareiss/Faddeev-Leverrier}
            \State $K \gets \Call{GetCanonical}{C}$
            \State $\mathcal{H}[K] \gets \mathcal{H}[K] + N_A$ \Comment{Aggregate counts for the equivalence class}
        \EndIf
    \EndFor
\EndFor
\State \Return $\mathcal{H}$
\end{algorithmic}
\end{algorithm}

To efficiently store and compare the equivalence classes, we require a canonical form function $\mathcal{K}: \{0,1\}^{n \times n} \to \{0,1\}^{n \times n}$ such that $A \sim B \iff \mathcal{K}(A) = \mathcal{K}(B)$. We compute this by mapping the matrix equivalence problem to the graph isomorphism problem, utilizing the individualization-refinement paradigm \cite{McKay1981, McKay2014}.

\subsection{Bipartite Graph Mapping}\label{sub:bipartite}

Let $A \in \{0,1\}^{n \times n}$ be a binary matrix. We associate $A$ with a bipartite graph $\Gamma(A) = (V, E)$, where the vertex set $V = R \cup C$ consists of row indices $R = \{r_1, \dots, r_n\}$ and column indices $C = \{c_1, \dots, c_n\}$. The edge set is defined by $E = \{ (r_i, c_j) \mid A_{ij} = 1 \}$. 

The equivalence relation $A = P B Q$ (under row and column permutations) corresponds to the isomorphism of $\Gamma(A)$ and $\Gamma(B)$ under permutations that preserve the ordered partition $(R, C)$. To account for the transpose case $A = P B^T Q$, we must also consider isomorphisms that map $R \to C$ and $C \to R$. 

We define the canonical form $\mathcal{K}(A)$ as the lexicographically minimal adjacency matrix obtained by reading the bits of the bipartite graph sequentially. To handle the transpose, we compute the canonical labeling twice: once starting with the ordered vertex partition $(R, C)$, and once with $(C, R)$. The global canonical form is the lexicographical minimum of these two results.

\begin{algorithm}[t]
\caption{Canonical Form via Individualization-Refinement}
\label{alg:canonical}
\begin{algorithmic}[1]
\Procedure{GetCanonical}{$C$}
    \State $\Gamma \gets \text{BipartiteGraph}(C)$ \Comment{Vertices $V = R \cup C$ (rows and columns)}
    \State $K_{RC} \gets \Call{SearchTree}{\Gamma, (R, C), \varnothing, 0}$
    \State $K_{CR} \gets \Call{SearchTree}{\Gamma, (C, R), \varnothing, 0}$
    \State \Return $\min(K_{RC}, K_{CR})$
\EndProcedure
\Statex
\Procedure{SearchTree}{$\Gamma, \pi, G, \text{depth}$}
    \State $\pi_{\text{eq}} \gets \Call{1-WL}{\Gamma, \pi}$
    \If{$\pi_{\text{eq}}$ is completely discrete}
        \State $M \gets \text{AdjacencyMatrix}(\Gamma, \pi_{\text{eq}})$
        \If{$M < M_{\text{best}}$}
            \State $M_{\text{best}} \gets M$
            \State $\pi_{\text{best}} \gets \pi_{\text{eq}}$
        \ElsIf{$M == M_{\text{best}}$}
            \State $\gamma \gets \text{ExtractAutomorphism}(\pi_{\text{eq}}, \pi_{\text{best}})$
            \State $G \gets \langle G, \gamma \rangle$ \Comment{Add new generator to the automorphism group}
        \EndIf
        \State \Return
    \EndIf
    \State $W \gets \text{First non-trivial cell in } \pi_{\text{eq}}$
    \State $\mathcal{O} \gets \text{Orbits}(W, G^{(\text{depth})})$ \Comment{Group vertices by current stabilizer subgroup}
    \For{each representative $v \in \mathcal{O}$} 
        \State $\pi' \gets \text{Individualize}(\pi_{\text{eq}}, v)$ \Comment{Split $v$ into its own singleton cell}
        \State \Call{SearchTree}{$\Gamma, \pi', G, \text{depth} + 1$}
    \EndFor
    \State \Return $M_{\text{best}}$
\EndProcedure
\end{algorithmic}
\end{algorithm}

\subsection{Partition Refinement (1-WL)}\label{sub:partitionrefinement}

To reduce the search space, we apply the 1-dimensional Weisfeiler-Leman (1-WL) algorithm \cite{weisfeiler1968reduction}, also known as color refinement, to transform an ordered partition $\pi$ of the vertex set into an equitable partition $\mathcal{R}(\Gamma, \pi)$.

\begin{definition}[Equitable Partition \cite{McKay1981}]
A partition $\pi$ is equitable with respect to a graph $\Gamma$ if for all cells $V_i, V_j \in \pi$ and all vertices $u, v \in V_i$, the number of neighbors of $u$ in $V_j$ equals the number of neighbors of $v$ in $V_j$.
\end{definition}

The refinement procedure $\mathcal{R}$ iteratively splits cells based on vertex degrees relative to other cells until the partition is equitable. It is a well-established result that this refinement procedure is a canonical invariant; that is, any isomorphism mapping $\Gamma(A)$ to $\Gamma(B)$ maps $\mathcal{R}(\Gamma(A), \pi)$ to $\mathcal{R}(\Gamma(B), \pi^\gamma)$ \cite{McKay1981}.

\subsection{Search Tree and Automorphism Pruning}\label{sub:searchtree}

Because equitable partitions are not necessarily discrete, refinement alone cannot distinguish all vertices. We disambiguate symmetries using a backtracking search tree $T(\Gamma, \pi)$. Each node $\nu$ in the tree represents a partition. Child nodes are generated by fixing a vertex $v$ from the first non-trivial cell and re-applying the refinement procedure $\mathcal{R}$.

To make this search computationally tractable, we employ the automorphism pruning strategy utilized by the \texttt{nauty} and \texttt{Traces} software packages \cite{McKay2014}. During the traversal of the search tree, whenever a leaf node yields an adjacency matrix identical to the current lexicographical best, an automorphism of the graph is discovered. We maintain a generating set for the automorphism group $G = \operatorname{Aut}(\Gamma)$ discovered during the search.

\begin{lemma}[Pruning Lemma \cite{McKay2014}]\label{lem:pruning-lemma}
Let $\nu_1$ and $\nu_2$ be nodes in the search tree $T$, where $\nu_1$ is visited before $\nu_2$ in the lexicographical search order. If there exists a discovered automorphism $\gamma \in G$ such that $\nu_2 = \nu_1^\gamma$, then the subtree rooted at $\nu_2$ contains no canonical candidates lexicographically smaller than those in the subtree of $\nu_1$, and may be safely discarded.
\end{lemma}

\subsection{Group-Theoretic Computations}\label{sub:group}

The efficiency of Lemma \ref{lem:pruning-lemma} relies on the ability to quickly compute orbits of the stabilizer subgroup without enumerating all group elements. Following standard computational group theory practices \cite{dixon1996permutation}, we represent the automorphism group $G$ using a stabilizer chain. 

When branching on a cell $W \in \pi$ at depth $k$ in the search tree, we compute the orbits of $W$ under the current stabilizer subgroup $G^{(k)}$. This is done efficiently using a Union-Find data structure on the Schreier generator graph \cite[Section 3.6]{dixon1996permutation}. By branching on only one representative vertex from each orbit, we drastically prune symmetric branches of the search tree. New automorphisms discovered at the leaves are integrated into the stabilizer chain in polynomial time.

\subsection{Computing Canonical Matrices}

Combining the ingredients of Subsections \ref{sub:bipartite}, \ref{sub:partitionrefinement}, \ref{sub:searchtree}, and \ref{sub:group}, we describe our procedure for computing canonical matrices in Algorithm \ref{alg:canonical}.

\begin{theorem}
The algorithm correctly computes the canonical form $\mathcal{K}(A)$, returning the exact lexicographical minimum of the equivalence class of $A$ under row permutations, column permutations, and transposition.
\end{theorem}
\begin{proof}
The backtracking search tree exhaustively covers the permutation space $S_n \times S_n$. The 1-WL refinement restricts this space without eliminating valid isomorphisms \cite{McKay1981}. Furthermore, the orbit pruning strategy is mathematically guaranteed to preserve at least one path to the lexicographically minimal leaf \cite[Theorem 2.15]{McKay2014}. By executing this search for both initial partitions $(R, C)$ and $(C, R)$ and taking the minimum, the algorithm correctly accounts for the transpose equivalence, ensuring the true global minimum of the equivalence class is found.
\end{proof}

\subsection{Exact Counts for $n \le 9$}\label{sub:exact}

Tables \ref{tab:exact_counts} and \ref{tab:exact_counts_9x9} present the exact number of strongly non-singular $n \times n$ binary matrices, $N_{n,k}$, parameterized by the number of non-zero entries $k$, for dimensions $n \in \{1, \dots, 9\}$. The full dataset in JSON format is hosted at \cite{OrellanaMateo2026Repo}. We note that, by computing a representative for each equivalence class and the count of each class, one may compute any statistics about the set, not just non-zero counts. The interested user can do so simply by running the publicly available code detailed in this section and available at \cite{OrellanaMateo2026Repo}.

\begin{table}[htbp]
    \centering
    {\footnotesize
    \begin{tabular}{@{}rrr@{\hspace{5em}}rrr@{\hspace{5em}}rrr@{}}
        \toprule
       \textbf{$n$} & \textbf{$k$} & \textbf{Count} &  \textbf{$n$} &  \textbf{$k$} & \textbf{Count} &  \textbf{$n$} & \textbf{$k$} & \textbf{Count} \\ 
        \midrule
        1 & 1 & 1 & 6& 21 & 270526947& 8& 10 & 3514  \\
        2 & 2 & 1 & 6& 22 & 188703604& 8& 11 & 97244  \\
         2 & 3 & 3 & 6& 23 & 115926816 & 8& 12 & 1986607   \\
        3& 3 & 1 & 6& 24& 58314572 & 8& 13 & 31504564  \\
        3& 4 & 9 & 6& 25& 25124264 & 8& 14 & 395478728   \\
        3& 5 & 29 & 6& 26& 8586990 & 8& 15 & 3943148852  \\
        3& 6 & 19 & 6& 27& 2411996 & 8& 16 & 31369251177   \\
        3& 7 & 10 & 6& 28& 484830 & 8& 17 & 203223435056  \\
        4& 4 & 1 & 6& 29& 72648 & 8& 18 & 1084492529346  \\
        4& 5 & 18 & 6& 30 & 6484 & 8& 19 & 4890576176528  \\
        4& 6 & 143 & 6& 31 & 396 & 8& 20 & 18805194537373   \\
        4& 7 & 592 & 7& 7 & 1& 8& 21 & 63047260972722  \\
        4& 8 & 1061 & 7& 8 & 63& 8& 22 & 185050195830538  \\
        4& 9 & 1380 & 7& 9 & 1960& 8& 23 & 484312155800218  \\
        4& 10 & 961 & 7& 10 & 39830 & 8& 24 & 1130001175634995   \\
        4& 11 & 644 & 7& 11 & 585347 & 8& 25 & 2388141629501388   \\
        4& 12 & 174 & 7& 12 & 6458571& 8& 26 & 4557722772711510  \\
        4& 13 & 34 & 7& 13 & 53744776 & 8& 27 & 7957465424434324  \\
        5 & 5& 1& 7& 14 & 337996072& 8& 28 & 12664141178519307  \\
        5 & 6& 30& 7& 15 & 1675103574& 8& 29 & 18570778015470134  \\
        5 & 7& 425 & 7& 16 & 6614988082& 8& 30 & 24957086783562282  \\
        5&  8& 3640 & 7& 17 & 21743417216& 8& 31 & 31061367927548198  \\
        5& 9& 19512 & 7&  18 & 59637219644 & 8& 32 & 35567140229358595  \\
        5& 10& 62018 & 7& 19 & 141545976180& 8& 33 & 37814938253601916  \\
        5& 11& 144352 & 7& 20 & 288911603858& 8& 34 & 37074021913720550   \\
        5& 12& 232426 & 7& 21 & 521858329294 & 8& 35 & 33779665724027360  \\
        5& 13& 308685 & 7& 22 & 827024639200 & 8& 36 & 28393427281756235  \\
        5& 14& 306022  & 7& 23 & 1176966346333 & 8& 37 & 22180209800850350  \\
        5& 15& 246385  & 7& 24 & 1482186878151 & 8& 38 & 15968137418716642  \\
        5& 16& 160032  & 7& 25 & 1693609745344 & 8& 39 & 10660173316494306  \\
        5& 17& 81122  & 7& 26 & 1724441624018 & 8& 40 & 6535322648553720  \\
        5& 18& 29238 & 7& 27 &1594098556619 & 8& 41 & 3700038252393912  \\
        5& 19& 8074 & 7& 28 & 1317019173583 & 8& 42 & 1911489482762620  \\
        5& 20& 1154 & 7& 29 & 988559392212 & 8& 43 & 905530224323564  \\
        5& 21& 116 & 7& 30 & 661909038532 & 8& 44 & 388067409532546  \\
        6& 6& 1 & 7 & 31 & 400500045128 & 8& 45 & 150995283933710  \\
        6& 7& 45 & 7& 32 & 214045604384 & 8& 46 & 52554074390450  \\
        6& 8& 985 & 7& 33 & 101997188760 & 8& 47 & 16368776680422  \\
        6& 9 & 13745 &7& 34 & 42257609448 & 8& 48 & 4468830973592 \\
        6& 10 & 132987 & 7& 35 & 15337468014 & 8& 49 & 1067165274292 \\
        6& 11& 900789 & 7& 36 & 4735733876 & 8& 50 & 217193007336 \\
        6& 12& 4226503 & 7& 37 & 1234197450 & 8& 51 & 37312365232  \\
        6& 13& 15051553 & 7& 38 & 261591224 & 8& 52 & 5233480948 \\
        6& 14& 40635867 & 7& 39 & 44391892 & 8& 53 & 591886560 \\
        6& 15& 90221539 & 7& 40 & 5601004 & 8& 54 & 51707652 \\
        6& 16& 161124379 & 7& 41 & 536708 & 8& 55 & 3482368 \\
        6& 17& 247776435 & 7& 42 & 32972 & 8& 56 & 156888 \\
        6& 18& 313666787 & 7& 43 & 1352 & 8& 57 & 4616 \\
        6& 19& 351245123 & 8& 8 & 1& &  &   \\
        6& 20& 329076355 & 8& 9 & 84& &  &   \\
        \bottomrule
    \end{tabular}
   }
   \vspace{3 mm}
    \caption{The exact values of $N_{n,k}$ for $n \in \{1,\ldots,8\}$}
    \label{tab:exact_counts}
\end{table}

\begin{table}[t]
    \centering
    {\footnotesize
    \begin{tabular}{@{}rr@{\hspace{5em}}rr@{\hspace{5em}}rr@{}}
        \toprule
        \textbf{$k$} & \textbf{Count} & \textbf{$k$} & \textbf{Count} & \textbf{$k$} & \textbf{Count} \\ 
        \midrule
        9 & 1 & 31 & 96719476982407744092 & 53 & 83333465510680978826 \\
        10 & 108 & 32 & 186503724910597960072 & 54 & 40330856830106979690 \\
        11 & 5838 & 33 & 335422734825487477468 & 55 & 18193563381572837212 \\
        12 & 210252 & 34 & 562215751430485988940 & 56 & 7606822625464081816 \\
        13 & 5652213 & 35 & 882789394906258039156 & 57 & 2945819219466818342 \\
        14 & 120044988 & 36 & 1296704229926641416276 & 58 & 1049968464418967194 \\
        15 & 2072228096 & 37 & 1790188101242490383464 & 59 & 343639220977425084 \\
        16 & 29407465584 & 38 & 2318475316973368470040 & 60 & 102439451887121968 \\
        17 & 344097864320 & 39 & 2828243308508161667208 & 61 & 27695978794734600 \\
        18 & 3333236753396 & 40 & 3243012884297962327568 & 62 & 6722470912110816 \\
        19 & 27016651978056 & 41 & 3507731778243932969769 & 63 & 1455146825624512 \\
        20 & 184889958771428 & 42 & 3570249044441509661900 & 64 & 277089814762980 \\
        21 & 1085361104090754 & 43 & 3430496716635991967290 & 65 & 45966172670268 \\
        22 & 5517513084920674 & 44 & 3103130069080774429632 & 66 & 6526518924368 \\
        23 & 24656256324744968 & 45 & 2649888111684432089913 & 67 & 783702562020 \\
        24 & 97542620062118552 & 46 & 2130133085694111554620 & 68 & 77872212440 \\
        25 & 345944420050724384 & 47 & 1615543159590023404340 & 69 & 6345678472 \\
        26 & 1104797971319879050 & 48 & 1152109135371983825440 & 70 & 409749896 \\
        27 & 3209822652195236944 & 49 & 774135602738705469776 & 71 & 20609608 \\
        28 & 8504988786860401576 & 50 & 488154989553138214988 & 72 & 711944 \\
        29 & 20726783710640352922 & 51 & 289298664178677485460 & 73 & 15760 \\
        30 & 46495824680544165784 & 52 & 160449071628568146100 & & \\
        \bottomrule
    \end{tabular}
   }
      \vspace{3 mm}
    \caption{The exact values of $N_{n,k}$ for $n =9$}
    \label{tab:exact_counts_9x9}
\end{table}

\section{An Upper Bound for Strongly Non-Singular $\{0,1\}$-Matrices}\label{sec:experiments}

Recall $N_n(p) =  \mathbb{P}\left[M_n(\mathrm{Ber}(p)) \text{ is strongly non-singular}\right]$ and, using Tables \ref{tab:exact_counts} and \ref{tab:exact_counts_9x9}, is exactly computable for $n \le 9$ (see Figure \ref{fig:exact_Nnp}). In particular, 
\[
N_n(p)= \sum_{k=0}^{n^2} N_{n,k} p^k (1-p)^{n^2-k},
\]
where $N_{n,k}$ is the number of $n\times n$ strongly non-singular $\{0,1\}$-matrices with exactly $k$ ones. For example, the total number of $9\times9$ strongly non-singular $\{0,1\}$-matrices is
\[
\sum_{k=9}^{73} N_{9,k} = 36,646,054,311,185,413,881,216,
\]
and therefore 
\[N_{9}(1/2) = \frac{36,646,054,311,185,413,881,216}{2^{81}} \approx 0.015156.\]
Here, using the data $N_{n,k}$ for $n = 9$ from Table \ref{tab:exact_counts_9x9}, we produce an upper bound for $N_n(p)$ for $n > 9$.

\begin{theorem}\label{thm:upper}
Let $n > 9$ and $p \in (0,1)$. Define $q = 1-p$, $\alpha = \min\{p,q\}$, and $\beta = \max\{p,q\}$. Then
\[N_n(p) \le \sum_{w=9}^{73} N_{9,w} p^{w}q^{81-w} \prod_{k \in T_p} \left[1 - \big(2q^k -q^{2k-1} +2 \big((p^2 + q^2) \alpha^{k-10} - q^k \beta^{k-10} \big) g(w) - \beta^{2k-19} g(w)^2\big) \right],\]
where $g(w) = 9 p^{\frac{w}{9}} q^{9-\frac{w}{9}}$ and
\[T_p = \left\{ k \in \mathbb{N} \,:\, 10\le k \le n,\; (p^2 + q^2) \alpha^{k-10} \ge q^k \beta^{k-10} + (p + 8 \alpha) \beta^{2k-11} \right\}.\]
\end{theorem}

\begin{proof}
Let $M:= M_n(\mathrm{Ber}(p))$ and let $M_k$ denote the leading $k \times k$ submatrix of $M$. To compute an upper bound for $N_n(p)$ for $n > 9$, we condition on the realization of $M_9$, and utilize data regarding its non-zero entries. We produce an upper bound for the weaker condition that $M_9$ is strongly non-singular, $M$ has no zero columns or zero rows, and that none of the first nine rows and columns of $M$ are repeated. Let $S$ be the set of all $9\times 9$ strongly non-singular matrices, $\bm{r}_k$ and $\bm{c}_k$ be the $k^{th}$ row and $k^{th}$ column, respectively, of $M_k$, and \[\mathcal{E}_k = \{ \bm{r}_k = \bm{0}\} \cup \{ \bm{c}_k = \bm{0}\}\cup \bigcup_{i=1}^9 \{ \bm{r}_k = \bm{r}_i\} \cup \bigcup_{i=1}^9 \{ \bm{c}_k = \bm{c}_i\}. \]

For a fixed $A\in S$, strong non-singularity implies $\mathcal E_k^c$ for every $k\ge10$. Therefore,
\[
\mathbb{P}\left[M\text{ is strongly non-singular}\mid M_9=A\right] \le \prod_{k=10}^n \left(1-\mathbb{P}\left[\mathcal E_k \mid M_9=A,\mathcal E_{10}^c,\ldots,\mathcal E_{k-1}^c \right] \right)
\]
Consequently,
\[
N_n(p) \le \sum_{A \in S} \mathbb{P}[M_9 = A] \prod_{k = 10}^n \left(
1- \mathbb{P}\left[ \mathcal{E}_k \mid M_9 = A, \mathcal{E}_{10}^c,\ldots, \mathcal{E}_{k-1}^c \right] \right)
\]
and so it suffices to lower bound
\[
\mathbb{P}\left[ \mathcal{E}_k \mid M_9 = A, \mathcal{E}_{10}^c,\ldots, \mathcal{E}_{k-1}^c \right]
\]
We do so by inclusion-exclusion, and note that no three of the events $\{ \bm{r}_k = \bm{0}\}$, $\{ \bm{c}_k = \bm{0}\}$, $\{ \bm{r}_k = \bm{r}_i\}$, $i = 1,\ldots, 9$, and  $\{ \bm{c}_k = \bm{c}_i\}$, $i = 1,\ldots,9$ can occur simultaneously, as $A$ has distinct non-zero rows and columns, and $\bm{c}_k$ and $\bm{r}_k$ can each take only one value. In addition, events $\{\bm{r}_k = \bm{r}_i\}$ and $\{\bm{r}_k = \bm{r}_j\}$ are disjoint for $i,j \in \{1,\ldots,9\}$, $i \ne j$, as $A$ has independent rows (and similarly for column events). Let $ \mathcal{\hat E}_{k}$ be the event $M_9 = A$ and $\mathcal{E}_{10}^c,\ldots,\mathcal{E}_{k}^c$. Let $\rho_i$ and $\sigma_i$ denote the number of non-zero entries in row and column $i$ of $A$, respectively. Then
\begin{align*}
\mathbb{P}\left[\bm{r}_k = \bm{0} \, | \, \mathcal{\hat E}_{k-1}\right] &= \mathbb{P}\left[\bm{c}_k = \bm{0} \, | \, \mathcal{\hat E}_{k-1}\right] = q^k, \\
\mathbb{P}\left[\bm{r}_k = \bm{r}_i \, | \, \mathcal{\hat E}_{k-1}\right] &\ge (p^2 + q^2) \alpha^{k-10} p^{\rho_i} q^{9-\rho_i}, \\
\mathbb{P}\left[\bm{c}_k = \bm{c}_i \, | \, \mathcal{\hat E}_{k-1}\right] &\ge (p^2 + q^2) \alpha^{k-10} p^{\sigma_i} q^{9-\sigma_i},\\
\mathbb{P}\left[\bm{r}_k = \bm{0}, \bm{c}_k = \bm{0} \, | \, \mathcal{\hat E}_{k-1}\right] &= q^{2k-1}, \\
\mathbb{P}\left[\bm{r}_k = \bm{0}, \bm{c}_k = \bm{c}_i \, | \, \mathcal{\hat E}_{k-1}\right] &\le q^{k} \beta^{k-10} p^{\sigma_i} q^{9-\sigma_i}, \\
\mathbb{P}\left[\bm{r}_k = \bm{r}_i, \bm{c}_k = \bm{0} \, | \, \mathcal{\hat E}_{k-1}\right] &\le q^{k} \beta^{k-10} p^{\rho_i} q^{9-\rho_i}, \\
\mathbb{P}\left[\bm{r}_k = \bm{r}_i, \bm{c}_k = \bm{c}_j \, | \, \mathcal{\hat E}_{k-1}\right] &\le \beta^{2k-19} p^{\rho_i + \sigma_j} q^{18-\rho_i -\sigma_j},
\end{align*}
where our inequalities are the result of the following analysis:
\begin{itemize}
\item $\{\bm{r}_k = \bm{r}_i \}$ ($\{\bm{c}_k = \bm{c}_i \}$ is similar): At step $k$, the target row $\bm{r}_i$ consists of the fixed $i^{th}$ row of $A$, a continuation in the $10^{th}$ to $(k-1)^{th}$ entries, and the new entry $M_{i,k}$. The new row $\bm{r}_k$ matches $\bm{r}_i$ if the first nine entries of $\bm{r}_k$ match the $i^{th}$ row of $A$, a probability $p^{\rho_i} q^{9-\rho_i}$ event, the middle $k-10$ entries match, at least an $\alpha^{k-10}$ probability event, and the last entry matches, a probability $p^2 + q^2$ event.
\item  $\{\bm{r}_k = \bm{0}, \bm{c}_k = \bm{c}_i\}$ ($\{\bm{r}_k = \bm{r}_i, \bm{c}_k = \bm{0}\}$ is similar): The event $\{\bm{r}_k = \bm{0}\}$ has probability $q^{k}$, and, conditioning on $M_{k,i} = M_{kk} = 0$, the probability of $\{\bm{c}_k = \bm{c}_i\}$ is at most $p^{\sigma_i}q^{9-\sigma_i} \beta^{k-10}$.
\item$\{\bm{r}_k = \bm{r}_i, \bm{c}_k = \bm{c}_j\}$: Both events holding requires $M_{k,k} = A_{i,j}$, a probability at most $\beta$ event. The first nine entries of $\bm{r}_k$ and $\bm{c}_k$ match $ \bm{r}_i$ and $ \bm{c}_j$ with probability $p^{\rho_i + \sigma_j} q^{18-\rho_i-\sigma_j}$. The remaining $k-10$ entries of $\bm{r}_k$ and $\bm{c}_k$ match $ \bm{r}_i$ and $ \bm{c}_j$ with probability at most $\beta^{2(k-10)}$.
\end{itemize}
Now, let $R(\rho) = \sum_{i=1}^9 p^{\rho_i} q^{9-\rho_i}$ and $C(\sigma) = \sum_{i=1}^9 p^{\sigma_i} q^{9-\sigma_i}$. We obtain the lower bound
\[
    \mathbb{P}[ \mathcal{E}_k \, | \, \mathcal{\hat E}_{k-1}] \ge 2q^k +(p^2 + q^2) \alpha^{k-10} \left(R(\rho)+C(\sigma)\right) -q^{2k-1}  - q^k \beta^{k-10}\left(R(\rho) +C(\sigma)\right) - \beta^{2k-19} R(\rho) C(\sigma)
\]

Note that $R(\rho)$ and $C(\sigma)$ are both upper bounded by $W:=(p + 8 \alpha) \beta^8$. To see this, let $f(t)=p^tq^{9-t}$. If $p\le q$, then $f$ is decreasing, and since every row of $A$ is nonzero, $\rho_i\ge1$ for all $i$, so
\[R(\rho)=\sum_{i=1}^9 f(\rho_i)\le 9f(1)=9pq^8=(p+8\alpha)\beta^8.\]
If $p\ge q$, then $f$ is increasing, and since the rows of $A$ are pairwise distinct, at most one row can have $\rho_i=9$ while the remaining eight rows have $\rho_i\le8$, so
\[R(\rho)\le f(9)+8f(8)=p^9+8p^8q=(p+8\alpha)\beta^8.\]
The same argument applies to $C(\sigma)$. Let
\[F_k(x,y):=2q^k-q^{2k-1}+\big((p^2+q^2)\alpha^{k-10}-q^k\beta^{k-10}\big)(x+y)-\beta^{2k-19}xy.\]
The lower bound above says
\[\mathbb{P}\left[\mathcal E_k\mid \mathcal{\hat E}_{k-1}\right] \ge F_k(R(\rho),C(\sigma)).\]
If $k\in T_p$, then the defining inequality for $T_p$ gives
\[(p^2+q^2)\alpha^{k-10}-q^k\beta^{k-10} \ge (p+8\alpha)\beta^{2k-11}=\beta^{2k-19}W.\]
Thus, for $0\le x,y\le W$,
\begin{align*}
    \frac{\partial F_k}{\partial x}(x,y)&=\big((p^2+q^2)\alpha^{k-10}-q^k\beta^{k-10}\big)-\beta^{2k-19}y\ge0, \\
    \frac{\partial F_k}{\partial y}(x,y)&=\big((p^2+q^2)\alpha^{k-10}-q^k\beta^{k-10}\big)-\beta^{2k-19}x\ge0.
\end{align*}
Hence $F_k$ is non-decreasing in each variable on $[0,W]^2$.

Now let $w:=\mathrm{nnz}(A)$. The function $f(r)=p^r q^{9-r}$ is convex, and $\rho_1+\cdots+\rho_9=w$; therefore Jensen's inequality gives $R(\rho)\ge g(w)$. Similarly, since $\sigma_1+\cdots+\sigma_9=w$, we have $C(\sigma)\ge g(w)$. Because $R(\rho),C(\sigma)\le W$ and $F_k$ is nondecreasing in each variable for $k\in T_p$, we obtain
\[\mathbb{P}\left[\mathcal E_k\mid \mathcal{\hat E}_{k-1}\right] \ge F_k(g(w),g(w)),\]
i.e.
\[\mathbb{P}\left[\mathcal E_k\mid \mathcal{\hat E}_{k-1}\right] \ge 2q^k -q^{2k-1} +2\big((p^2 + q^2) \alpha^{k-10} - q^k \beta^{k-10}\big) g(w) -\beta^{2k-19}g(w)^2.\]
Since each factor $1-\mathbb{P}[\mathcal E_k\mid \mathcal{\hat E}_{k-1}]$ lies in $[0,1]$, omitting the factors with $k\notin T_p$ can only increase the product. Grouping the matrices $A\in S$ by their number $w$ of non-zero entries gives the stated bound.\end{proof}

We plot the upper bound of Theorem \ref{thm:upper}  and experimental results for $n = 30$ in Figure \ref{fig:upper_exp}. The above bound simplifies significantly for $p = 1/2$, the case of greatest interest.

\begin{corollary}\label{cor:1/2}
Let $n >9$. Then
\[N_n(1/2) \le \frac{N_9}{2^{81}} \prod_{k = 10}^n \left(1 - \frac{20}{2^{k}} +\frac{200}{2^{2k}}\right).\]
\end{corollary}
\begin{proof}
When $p = 1/2$, $g(w) = 9 \times 2^{-9}$ for every $w$, and $T_p = \{10,\ldots,n\}$. Because $p = q = \alpha = \beta = 1/2$, each factor in the product simplifies as
\[1 - \big(2q^k -q^{2k-1} +2 \big((p^2 + q^2) \alpha^{k-10} - q^k \beta^{k-10} \big) g(w) - \beta^{2k-19} g(w)^2\big)=1 - \frac{20}{2^{k}} +\frac{200}{2^{2k}}.\]
\end{proof}

For $n = 30$ (the dimension used in Figure \ref{fig:upper_exp}), Corollary \ref{cor:1/2} gives the reported upper bound of approximately $1.45\ldots \%$.

\section*{Acknowledgements}

This project was conducted as part of the MIT Summer Research Program (MSRP) during the summer of 2025. The authors thank Louisa Thomas for improving the style of presentation, and Anna Brandenburger and Byron Chin for their insightful discussions. We would like to thank Mehtaab Sawhney for his suggestions regarding proof techniques for Corollary \ref{corr:main} and Dmitrii Zakharov for pointing us to Elkes-Szab{\'o} type theorems, which proved invaluable for proving Lemma \ref{lm:det_grid}. Dmitrii also suggested that an alternate technique, involving the Szemerédi–Trotter theorem, may be possible. This material is based upon work supported by the National Science Foundation under grant no. DMS-2513687.

{ \small 
	\bibliographystyle{plain}
	\bibliography{main.bib}
    
}

\end{document}